\documentclass[sn-mathphys,Numbered]{sn-jnl}
\usepackage{adjustbox}
\usepackage{graphicx}%
\usepackage{multirow}%
\usepackage{amsmath,amssymb,amsfonts}%
\usepackage{amsthm}%
\usepackage{mathrsfs}%
\usepackage[title]{appendix}%
\usepackage{xcolor}%
\usepackage{textcomp}%
\usepackage{manyfoot}%
\usepackage{booktabs}%
\usepackage{algorithm}%
\usepackage{algorithmicx}%
\usepackage{algpseudocode}%
\usepackage{listings}%
\newcommand{\ba}{\begin{array}}
\newcommand{\ea}{\end{array}}
\newcommand{\be}{\begin{equation}}
\newcommand{\ee}{\end{equation}}
\newcommand{\bd}{\begin{displaymath}}
\newcommand{\ed}{\end{displaymath}}
\newcommand{\bi}{\begin{itemize}}
\newcommand{\ei}{\end{itemize}}
\newcommand{\bn}{\begin{enumerate}}
\newcommand{\en}{\end{enumerate}}
\newcommand{\pa}{\partial}
\newcommand{\f}{\frac}
\newcommand{\ci}{\cite}

\usepackage{amsmath,amssymb,amsfonts}

\usepackage{todonotes}

\theoremstyle{thmstyleone}%
\newtheorem{theorem}{Theorem}
\newtheorem{lemma}{Lemma}
\newtheorem{corollary}{Corollary} 

\theoremstyle{thmstyletwo}%
\newtheorem{remark}{Remark}%

\theoremstyle{thmstylethree}%

\begin{document}

\title[Hermite Methods for Maxwell's Equations]{Energy-Conserving Hermite Methods for Maxwell's Equations}

\author[1]{\fnm{Daniel} \sur{Appel\"o}}\email{appelo@vt.edu}

\author*[2]{\fnm{Thomas} \sur{Hagstrom}}\email{thagstrom@smu.edu}

\author[3]{\fnm{Yann-Meing} \sur{Law}}\email{yann-meing.law@csulb.edu}

\affil[1]{\orgdiv{Department of Mathematics}, \orgname{Virginia Tech}, \orgaddress{\city{Blacksburg}, \postcode{24060}, \state{VA}, \country{USA}}}

\affil[2]{\orgdiv{Department of Mathematics}, \orgname{Southern Methodist University}, \orgaddress{\city{Dallas}, \postcode{75275}, \state{TX}, \country{USA}}}

\affil[3]{\orgdiv{Department of Mathematics and Statistics}, \orgname{California State University Long Beach}, \orgaddress{\city{Long Beach}, \postcode{90840},
\state{CA}, \country{USA}}}

\abstract{Energy-conserving Hermite methods for solving Maxwell's equations in dielectric and dispersive media are described and analyzed.
In three space
dimensions methods of order $2m$ to $2m+2$ require $(m+1)^3$ degrees-of-freedom per node for each field variable and can be explicitly marched in
time with steps independent of $m$. We prove stability for time steps limited only by domain-of-dependence requirements along with
error estimates in a special seminorm associated with the interpolation process.
Numerical experiments are presented which demonstrate that Hermite methods of very high order enable
the efficient simulation of electromagnetic wave propagation over thousands of wavelengths.} 



\maketitle

\section{Introduction}\label{sec:intro}

Hermite methods are general-purpose discretization schemes 
for solving time dependent partial differential equations exploiting the
unique projection properties of Hermite-Birkhoff interpolation \cite{CHIDESpaper}.
Hermite methods are
particularly well-suited for hyperbolic equations for two reasons:
\begin{itemize}
\item In contrast
with typical polynomial-based element methods, Hermite methods
for hyperbolic problems can march in time in interior domains with a time step, $\Delta t$, 
limited only by domain-of-dependence constraints, 
$c \Delta t \lesssim \Delta x$, independent of order.
\item The
cell updates require no communication with neighboring cells, and so 
high-order Hermite methods essentially maximize the computation-to-communication
ratio.
\end{itemize}

Examples
of the application of Hermite methods in the hyperbolic case include the original
dissipative formulation \ci{GoodHer} as well as more recent
energy-conserving forms \ci{HermWave,HermiteLF,Arturo_ICOS16}. The latter references also
include implementations on GPUs where the localization of the cell updates can be exploited.

Here we consider the general dispersive Maxwell system:
\begin{eqnarray*}
\epsilon \left(1 + \mathcal{K}_e \ast \right) \f {\pa E}{\pa t} & = & \nabla \times H ,  \\
\mu \left(1 + \mathcal{K}_m \ast \right) \f {\pa H}{\pa t}  & = & - \nabla \times E . 
\end{eqnarray*} 
We assume Lorentz models for the temporal convolutions; precisely, with $s$ the Laplace transform
variable dual to time, 
\begin{eqnarray*}
\hat{\mathcal{K}}_e & = & \sum_{j=1}^{N_e} \f {\omega_{e,j}^2}{s^2 + \gamma_{e,j} s + \Omega_{e,j}^2}, \\ 
\hat{\mathcal{K}}_m & = & \sum_{j=1}^{N_m} \f {\omega_{m,j}^2}{s^2 + \gamma_{m,j} s + \Omega_{m,j}^2} . 
\end{eqnarray*}
Here we include frequency dependence not only
of the permittivity but also of the permeability to account for simple homogenized models of metamaterials.
Note that more general models, as discussed in \ci{OpticalDisperse}, could also be treated, and applications
of the method to nonlinear dispersive media will appear in \ci{HermiteNLDisperse}. As our focus here is on
energy-conserving discretizations, we will consider cases where the dissipation can be neglected, $\gamma_{e,j}=\gamma_{m,j}=0$,
where the Lorentz model reduces to a so-called Sellmeier model. 
We eliminate the convolutions by introducing additional fields $K_j$, $L_j$, $R_j$ and $S_j$ to obtain:
\begin{eqnarray}
\f {\pa E}{\pa t} & = & \f {1}{\epsilon} \nabla \times H - \sum_{j=1}^{N_e} \omega_{e,j}^2 K_j , \label{Et} \\
\f {\pa K_j}{\pa t} & = & -\gamma_{e,j} K_j - \Omega_{e,j}^2 L_j + E , \label{Pjt} \\
\f {\pa H}{\pa t} & = & -\f {1}{\mu} \nabla \times E - \sum_{j=1}^{N_m} \omega_{m,j}^2 R_j , \label{Ht} \\
\f {\pa R_j}{\pa t} & = & -\gamma_{m,j} R_j - \Omega_{m,j}^2 S_j + H , \label{Rjt} \\
\f {\pa L_j}{\pa t} = K_j , & &
\f {\pa S_j}{\pa t} = R_j . \label{Sjt} 
\end{eqnarray} 
After rescaling the variables  we can rewrite (\ref{Et})-(\ref{Sjt})
in the form: 
\begin{eqnarray}
\f {\pa V}{\pa t} & = & \sum_k A_k \f {\pa W}{\pa x_k} + M W - \Gamma_V V , \label{Vt} \\
\f {\pa W}{\pa t} & = & \sum_k A_k^T \f {\pa V}{\pa x_k} - M^T V - \Gamma_W W , \label{Wt} 
\end{eqnarray}
with
\be
V = \left( \ba{c} \sqrt{\epsilon} E \\ \sqrt{\epsilon} \omega_{e,1} \Omega_{e,1} L_1 \\ \vdots \\
\sqrt{\epsilon} \omega_{e,N_e} \Omega_{e,N_e} L_{N_e} \\ \sqrt{\mu} \omega_{m,1} R_1 \\ \vdots \\
\sqrt{\mu} \omega_{m,N_m} R_{N_m} \ea \right) , \ \ \ \
W = \left( \ba{c} \sqrt{\mu} H \\ \sqrt{\mu} \omega_{m,1} \Omega_{m,1} S_1 \\ \vdots \\
\sqrt{\mu} \omega_{m,N_m} \Omega_{m,N_m} S_{N_m} \\ \sqrt{\epsilon} \omega_{e,1} K_1 \\ \vdots \\
\sqrt{\epsilon} \omega_{e,N_e} K_{N_e} \ea \right) . \label{VWdef}
\ee
Here, in $3 \times 3$ block form, the skew-symmetric matrices $A_k$ encode the curl operator
\bd
c \left( \ba{ccc} \nabla \times & 0 & 0 \\ 0 & 0 & 0 \\ 0 & 0 & 0 \ea \right) , \ \ \  
c=(\epsilon \mu)^{-1/2} ,
\ed
$M$ is given by
\bd
M= \left( \ba{ccc} 0 & 0 & -{\rm diag} (\omega_{e,j}) \\ 0 & 0 & {\rm diag} (\Omega_{e,j}) \\
{\rm diag} (\omega_{m,j}) & -{\rm diag}( \Omega_{m,j} ) \ea \right) .
\ed
and the dissipation matrices are nonnegative and diagonal,
\bd
\Gamma_V = {\rm diag}(0 \ \gamma_{e,j} \ 0), \ \ \ \Gamma_W = {\rm diag}(0 \ \gamma_{m,j} \ 0) .
\ed
Spatial
derivatives are only applied to $E$ and $H$ and the characteristic speeds are
$c$, $0$. We thus conclude that the domain-of-dependence, which is fundamental to
the application of Hermite methods, is
unaffected by the dispersive corrections. In addition, an energy given by $\| V \|_{L^2}^2 + \| W \|_{L^2}^2$ is
conserved or dissipated (modulo boundary contributions), and the number and type of
admissible boundary conditions is the same as for Maxwell's equations in a simple dielectric.

\section{Conservative Hermite Discretization of the Dispersive Maxwell System}\label{sec:method}

The essential ingredients of all Hermite methods are:
\begin{description}
\item[i.] A cuboidal primal and dual grid, 
\item[ii.] Degrees of freedom defined by tensor-product Taylor polynomials
at the cell vertices,
\item[iii.] Cell polynomials constructed as tensor-product Hermite-Birkhoff
interpolants of the vertex data,
\item[iv.] Local (cell-wise) evolution to produce updated degrees-of-freedom at
dual cell nodes.
\end{description}

Our focus here is on energy-conserving methods exploiting the special structure of the Maxwell system.
To that end we assume that $\gamma_{e,j}=\gamma_{m,j}=0$. In our subsequent discussion we will indicate
how the method can be modified to include dissipation. We note that the original dissipative Hermite method
analyzed in \ci{GoodHer} is directly applicable to the dispersive Maxwell system. However, the proposed, staggered
method is more efficient and in some cases the exact energy conservation may be a desired feature. If
dissipative models are used, however, the original method can be used at higher order than the method proposed here. 
We are assuming a uniform Cartesian mesh and piecewise
uniform media. Methods for treating mapped grids to accommodate smooth boundaries are straightforward to
implement and will be briefly discussed later on. We are also exploring the use of purely Cartesian meshes
and embedded boundaries \ci{Hermiteembed}. Denote the vertices on the primal cells
by $(x_{1,j_1},x_{2,j_2},x_{3,j_3})$ and on the dual cells by $(x_{1,j_1+1/2},x_{2,j_2+1/2},x_{3,j_3+1/2})$
and set $\Delta x_k = x_{k,j_k+1}-x_{k,j_k}=x_{k,j_k+1/2}-x_{k,j_k-1/2}$. 

We define $V$ and $W$ at different time levels and thus on different grids. Using the standard multiindex
notation we define the degrees-of-freedom to be
\begin{eqnarray}
V_{j_1,j_2,j_3}^{\alpha,h} (t_n) & \approx & \f {\Delta x^{\arrowvert \alpha \arrowvert}}{\alpha!} 
D^{\alpha} V(t_n) , \label{Vhdef} \\
W_{j_1+1/2,j_2+1/2,j_3+1/2}^{\alpha,h} (t_{n+1/2}) & \approx & \f {\Delta x^{\arrowvert \alpha \arrowvert}}{\alpha!} 
D^{\alpha} W(t_{n+1/2}) , \label{Whdef}
\end{eqnarray}
with
\bd
\alpha=(\alpha_1,\alpha_2,\alpha_3), \ \ \ 0 \leq \alpha_j \leq m, \ \ \ \arrowvert \alpha \arrowvert = \alpha_1 + \alpha_2 + \alpha_3 . 
\ed
To describe the numerical process assume we know $V_{j_1,j_2,j_3}^{\alpha,h}(t_n)$ and $W_{j_1 \pm 1/2,j_2 \pm 1/2,j_3 \pm 1/2}^{\alpha,h}(t_{n+1/2})$.
Our goal is to update $V$. The first step is to compute the tensor-product Hermite-Birkhoff interpolant of the $W$ data. Precisely
we determine the unique tensor-product vector-valued polynomial
\bd
\tilde{W}_{j_1,j_2,j_3} (x_1,x_2,x_3) = \sum_{k_1=0}^{2m+1} \sum_{k_2=0}^{2m+1} \sum_{k_3=0}^{2m+1} C_{k_1,k_2,k_3} (x_1-x_{1,j_1})^{k_1}
(x_2-x_{2,j_2})^{k_2} (x_3-x_{3,j_3})^{k_3} ,
\ed
satisfying the interpolation conditions 
\be
\f {\Delta x^{\arrowvert \alpha \arrowvert}}{\alpha!} D^{\alpha} \tilde{W}_{j_1,j_2,j_3} (\mathbf{x}_{j_1 \pm 1/2,j_2 \pm 1/2,j_3 \pm 1/2})
= W_{j_1 \pm 1/2,j_2 \pm 1/2,j_3 \pm 1/2}^{\alpha,h}(t_{n+1/2}) . \label{HermiteBirkhoff}
\ee
To evolve we choose $q$ and use the Taylor approximation
\bd
V(t_{n+1}) = V(t_n) + 2 \sum_{\ell=1}^q \f {\left( \Delta t/2 \right)^{2 \ell -1}}{(2 \ell-1)!} \f {d^{2 \ell-1} V}{dt^{2 \ell -1}} (t_{n+1/2}) .
\ed
The time derivatives can be recursively computed using only $\tilde{W}$:
\begin{eqnarray}
V^1 & = & \sum_k A_k \f {\pa \tilde{W}_{j_1,j_2,j_3}}{\pa x_k} + M \tilde{W}_{j_1,j_2,j_3} , \label{V1} \\ 
V^{\ell} & = & \left( \sum_k A_k \f {\pa}{\pa x_k} + M \right)
\left( \sum_k A_k^T \f {\pa V^{\ell-1}}{\pa x_k} - M^T V^{\ell-1} \right) . \label{Vell} 
\end{eqnarray}
We emphasize that the functions $V^{\ell}$ are all tensor-product polynomials. Thus the updated data
can be obtained by simply differentiating the temporal Taylor series in space:
\be
V_{j_1,j_2,j_3}^{\alpha,h} (t_{n+1}) = V_{j_1,j_2,j_3}^{\alpha,h} (t_n) + \f {\Delta x^{\arrowvert \alpha \arrowvert}}{\alpha!} D^{\alpha} 
\left( 2 \sum_{\ell=1}^q \f {\left( \Delta t/2 \right)^{2 \ell -1}}
{(2 \ell -1 )!} V^{\ell} \right)(\mathbf{x}_{j_1,j_2,j_3}) . \label{Vup}
\ee

The procedure for updating $W$ from $t_{n-1/2}$ to $t_{n+1/2}$ is completely analogous; we list the steps below
for completeness. First compute the interpolating polynomial $\tilde{V}_{j_1+1/2,j_2+1/2,j_3+1/2}$ satisfying
\begin{eqnarray}
\f {\Delta x^{\arrowvert \alpha \arrowvert}}{\alpha!} D^{\alpha} \tilde{V}_{j_1+1/2,j_2+1/2,j_3+1/2}
(\mathbf{x}_{j_1 +1/2 \pm 1/2,j_2 + 1/2 \pm 1/2, j_3 + 1/2 \pm 1/2})
& = & \nonumber \\
V_{j_1 +1/2 \pm 1/2,j_2 +1/2 \pm 1/2,j_3 +1/2 \pm 1/2}^{\alpha,h}(t_{n}) . \label{HermiteBirkhoff2}
\end{eqnarray}
Then compute time derivatives recursively:
\begin{eqnarray}
W^1 & = & \sum_k A_k^T \f {\pa \tilde{V}_{j_1+1/2,j_2+1/2,j_3+1/2}}{\pa x_k} - M^T \tilde{V}_{j_1+1/2,j_2+1/2,j_3+1/2} , \label{W1} \\ 
W^{\ell} & = & \left( \sum_k A_k^T \f {\pa}{\pa x_k} - M^T \right)
\left( \sum_k A_k \f {\pa W^{\ell-1}}{\pa x_k} + M W^{\ell-1} \right) . \label{Well} 
\end{eqnarray}
Finally update the solution data:
\begin{eqnarray}
W_{j_1+1/2,j_2+1/2,j_3+1/2}^{\alpha,h} (t_{n+1/2}) & = & W_{j_1+1/2,j_2+1/2,j_3+1/2}^{\alpha,h} (t_{n-1/2}) \label{Wup} \\
+ & \f {\Delta x^{\arrowvert \alpha \arrowvert}}{\alpha!} D^{\alpha} & 
\left( 2 \sum_{\ell=1}^q \f {\left( \Delta t/2 \right)^{2 \ell -1}}
{(2 \ell -1 )!} W^{\ell} \right) (\mathbf{x}_{j_1+1/2,j_2+1/2,j_3+1/2})  . \nonumber
\end{eqnarray} 

\subsection{Dissipative Corrections}\label{sec:diss}

To include the dissipation terms in the evolution we propose solving a differential equation for these terms using
an implicit Nordsieck method in predictor-corrector form \ci[Ch. III-6]{Hairer-WannerI}.
For example, consider corrections to (\ref{Vup}). Define $\tilde{D}_V$ and $\tilde{D}_{\tilde{W}}$ as
solutions to the differential equation:
\be
\f {\pa \tilde{D}_{V}}{\pa t} = -\Gamma_V \f {\pa V}{\pa t}, \ \ 
\f {\pa \tilde{D}_{\tilde{W}}}{\pa t} =-\Gamma_W \f {\pa \tilde{W}}{\pa t} . \label{DDe}
\ee 
Note that since the dissipation matrices are diagonal and equal to zero in many blocks these equations do not
involve all the variables. In addition, since the spatial and temporal interpolation operators commute, the
polynomial $\tilde{D}_{\tilde{W}}$ can be computed via the Hermite-Birkhoff spatial interpolations of polynomials
computed in the preceding update of $W$. Maintaining the stucture of (\ref{Vup}),
we must simply incorporate the additional terms in the formulas (\ref{V1})-(\ref{Vell}): 
\begin{eqnarray}
V^1 & = & \sum_k A_k \f {\pa \tilde{W}_{j_1,j_2,j_3}}{\pa x_k} + M \tilde{W}_{j_1,j_2,j_3} + \f {\pa \tilde{D}_{V,j_1,j_2,j_3}^{(p)}}{\pa t} ,
\label{V1diss} \\ 
V^{\ell} & = & \left( \sum_k A_k \f {\pa}{\pa x_k} + M \right)
\left( \sum_k A_k^T \f {\pa V^{\ell-1}}{\pa x_k} - M^T V^{\ell-1} \right) \label{Velldiss} \\
& & + \f {\pa^{2\ell-2} \tilde{D}_{\tilde{W},j_1,j_2,j_3}}{\pa t^{2 \ell -2}}
+ \f {\pa^{2 \ell-1} \tilde{D}_{V,j_1,j_2,j_3}^{(p)}}{\pa t^{2 \ell -1}} . \nonumber  
\end{eqnarray}
Here we note that we are using the predicted values of $\tilde{D}_V$. Since the Nordsieck form represents the solution as a polynomial in time,
the time derivatives can be directly computed. Also, since it is interpolated the terms involving $\tilde{D}_{\tilde{W}}$ will have the tensor-product
degree $2m+1$ while $\tilde{D}_V$ will only be of degree $m$. However, due to the shrinking stability domain of the Nordsieck methods with
increasing $p$, we limit the order used to represent the dissipative terms. Therefore, the formal temporal order of the method will now be less 
than the spatial order for large values of $m$. 

We remark that the implicit assumption in this procedure is that the dissipative corrections are small. Then we expect that their
inclusion will have a negligible effect on the time step stability constraints. 
Although we exclude these terms in our analysis we include them in one of the numerical examples. We then find that in some cases the order reduction is
significant and hence sometimes favors the use of lower values of $m$ than in the nondissipative cases. 

\subsection{Implementation in Mapped Coordinates and Compatiblity Conditions}
The methods we have proposed are most efficient for piecewise uniform media. In particular the recursions (\ref{V1})-(\ref{Vell})
and (\ref{W1})-(\ref{Well}) require significantly fewer operations when no differentiations of the coefficients are needed.
At boundaries and interfaces some modifications are required. As mentioned above, we are experimenting with embedded boundary and
interface methods \ci{Hermiteembed}. Although it is at this time unclear if that method can be extended to high order, it should be possible to
combine it with higher order methods away from boundaries and interfaces. The alternative is to use mapped cells where necessary and to
use the equations in conjunction with the interface conditions to extend the solution to ghost nodes. This approach is proposed in
\ci{HermiteCompat}. We note that one can choose to either define the component vectors in reference to a fixed Cartesian system or also
transform them using into components referenced to the mapped system as proposed for the Yee scheme in \ci{Yeemapped}. In either case
the only change to the method appears in the details of the recursions. As suggested in \ci{HermiteCompat}, using a representation of the
mapping as a Taylor polynomial of sufficiently high order centered in the cell, the additional cost involves the multiplication of
the derivatives of the field interpolants by the coefficients arising from the mapping. For example
\bd
\f {\pa V}{\pa x_k} = \sum_{j} \f {\pa r_j}{\pa x_k} \f {\pa V}{\pa r_j} ,
\ed
increasing the cost due to the three polynomial multiplications. Note that these multiplications
can be truncated according to the eventual truncation of the update.

The major complication in the implementation of Hermite methods is the imposition of boundary and interface conditions. This stems from the need to provide
normal derivative data to update the solution in the cells adjacent to the boundary. For the dissipative formulation three approaches have had
success:
\begin{description}
\item[i.] Coupling with discontinuous Galerkin discretizations in a possibly unstructured mesh near the boundary \ci{HermDG}. Here local time stepping
in the DG elements allows us to retain the large global time steps in most of the domain.
\item[ii.] The correction function method \ci{HermiteCorrection}. This involves a weighted least-squares construction of a space-time polynomial near
the boundary. Penalty terms in the least squares construction involve the boundary evolution, Maxwell's equations, and a match with the Hermite
evolution in nearby volume cells.
\item[iii.] Compatibility conditions \ci{HermiteCompat}. Here one uses the boundary conditions along with the equation and its normal and tangential derivatives
to compute the missing data required to evolve the polynomial at the boundary.
\end{description}

Of these methods, only the compatibility approach has been demonstrated to work with a conservative Hermite method, namely the scheme for the scalar
wave equation studied in \ci{HermWave}. As such it is not directly applicable to the Maxwell system studied here, though it is a promising avenue of
future research to explore its extension to the present case.
Consider the example of a flat PEC boundary $x_1={\rm constant}$. Then if the mesh containing the
magnetic field is aligned with the boundary the problem is to determine the electric field and its derivatives at a dual ghost node. This is easily
accomplished by assuming that the tangential fields are extended as odd functions and the normal field as an even function. Extending this procedure
to a curved boundary and mapped coordinates leads to an algebraic system enforcing a zero tangential field along the boundary and a zero normal derivative
of the projection of the electic field in the normal direction. In \ci{HermiteCompat} the
scalar wave equation is considered and sixth order convergence for a conservative Hermite scheme with $m=3$ is demonstrated.

\section{Stability and Convergence}\label{sec:converge}

To establish the stability and convergence of the proposed method we exploit
the projection property of the Hermite-Birkhoff interpolation process defined by (\ref{HermiteBirkhoff}), (\ref{HermiteBirkhoff2}) along
with standard interpolation error estimates. (See \ci{GoodHer,HermWave,HermiteLF} for detailed proofs.) 
We will assume throughout this section that the solution
is $2 \pi$-periodic in each Cartesian coordinate and denote the domain by $\mathbb{T}$. Denote by
$\mathcal{I}_m$ the interpolation operator; to cut down on the required notation we use the same symbol for interpolation on the
dual and primal grids. The essential property is expressed as the orthogonality of interpolants and interpolation errors in a
certain seminorm, which we will call the HB seminorm. Precisely, if we define for any vector functions $f$ and $g$ the semi-inner-product
\be
\langle f,g \rangle_m = \langle \f {\pa^{3m+3} f}{\pa x_1^{m+1} \pa x_2^{m+1} \pa x_3^{m+1}} ,
\f {\pa^{3m+3} g}{\pa x_1^{m+1} \pa x_2^{m+1} \pa x_3^{m+1}}
\rangle_{\left(L^2(\mathbb{T})\right)^3} , \label{sip}
\ee
then
\be
\langle \mathcal{I}_m f,g - \mathcal{I}_m g \rangle_m = 0 . \label{orthlem}
\ee
Denoting by $\arrowvert \cdot \arrowvert_m^2$ the seminorm associated with the semi-inner-product (\ref{orthlem}) implies the Pythagorean Theoerem:
\be
\arrowvert f \arrowvert_m^2 = \arrowvert \mathcal{I}_m f \arrowvert_m^2 + \arrowvert f - \mathcal{I}_m f \arrowvert_m^2 . \label{Pythag} 
\ee

We will focus on establishing stability and convergence for the case of a dielectric medium. Since the dispersive terms present themselves as zero order
perturbations to the dielectric system they are straightforward to include, at least suboptimally, once the principal order terms have been handled. The time-staggered
exact evolution satisfies an energy conservation law in any Sobolev seminorm, including the HB seminorm defined above. Expanding in a Fourier series
in space, let $\hat{E}(k,t)$, $\hat{H}(k,t)$ be the Fourier coefficients of the symmetrized variables
$\tilde{E}=\sqrt{\epsilon} E$, $\tilde{H}=\sqrt{\mu} H$. They satisfy the ordinary differential equations
\be
\f {\pa \hat{E}}{\pa t} = i c k \times \hat{H}, \ \ \f {\pa \hat{H}}{\pa t} = - i c k \times \hat{H} .  \label{dielectichat}
\ee
For $k \neq 0$ set $k= \arrowvert k \arrowvert \hat{k}$. We will make use of an orthogonal decomposition of the
fields $\tilde{E}=E_S+E_N$, $\tilde{H}=H_S+H_N$ defined for any vector function $U$ by:
\be
\widehat{U_N}(k)=\hat{k} \hat{k}^T \hat{U}(k), \ \ \widehat{U_S} (k) = \hat{U}(k) - \widehat{U_N}(k) . \label{Projdef}
\ee
We also define the operator $\mathcal{C}$ applied to any vector function $U$ by
\be
\widehat{\mathcal{C}U}(k) = \hat{k} \times \hat{U}(k) , \label{Cdef}
\ee
and note the identities which follow from elementary identities satisfied by the cross product
\be
\mathcal{C}U=\mathcal{C}U_S, \ \ \mathcal{C}^2 U_S = -U_S, \ \ \left\arrowvert \widehat{\mathcal{C} U_S}(k) \right\arrowvert
= \left\arrowvert \widehat{U_S}(k) \right\arrowvert . \label{Cident}
\ee
The last identity combined with Parseval's relation implies that $\mathcal{C}$ preserves all Sobolev norms of
$U_S$. 

We also define operators $S_{\pm}$ as in \ci{HermWave,HermiteLF}:
\be
\widehat{S^{\pm} U}(k) = e^{\pm ic \arrowvert k \arrowvert \Delta t/2} \hat{U}(k) , \label{Sdef}
\ee
noting that $S^{-}=S^{+,\ast}$ and that the operators are unitary,
\be
S^{+}S^{-}=S^{-}S^{+}=I. \label{SpmI}
\ee
In addition they commute with the operator $\mathcal{C}$.
Then the exact evolution formulas take the form:
\begin{eqnarray}
\tilde{E}(\mathbf{x},t+\Delta t) & = & \tilde{E}(\mathbf{x},t) + \left( S^{+}-S^{-} \right) \mathcal{C} \tilde{H} (\mathbf{x}, t+ \Delta t/2) , \label{Enew} \\
\tilde{H}(\mathbf{x},t+\Delta t/2) & = & \tilde{H}(\mathbf{x},t- \Delta t/2) - \left( S^{+}-S^{-} \right) \mathcal{C} \tilde{E} (\mathbf{x}, t) . \label{Hnew}
\end{eqnarray}
Rewriting these in terms of the orthogonal decomposition and utilizing (\ref{Cident}) we have
\begin{eqnarray}
E_S(\mathbf{x},t_+\Delta t) & = & E_S(\mathbf{x},t) + \left( S^{+}-S^{-} \right) \mathcal{C} H_S (\mathbf{x}, t+ \Delta t/2) , \label{ES} \\ 
H_S(\mathbf{x},t_+\Delta t/2) & = & H_S(\mathbf{x},t-\Delta t/2) - \left( S^{+}-S^{-} \right) \mathcal{C} E_S (\mathbf{x}, t) , \label{HS} \\ 
E_N(\mathbf{x},t_+\Delta t) & = & E_N(\mathbf{x},t) , \label{EN} \\
H_N(\mathbf{x},t_+\Delta t/2) & = & H_N(\mathbf{x},t-\Delta t/2) . \label{HN}
\end{eqnarray}
Assuming $\nabla \cdot E = \nabla \cdot H =0$ initially equations (\ref{EN})-(\ref{HN}) simply imply that the fields will be
solenoidal at all subsequent discrete times. We will assume this to be true when estimating the errors. Setting
\begin{eqnarray}
P^{\pm} (\mathbf{x},t) & = & E_S (\mathbf{x},t) \mp S^{\pm} \mathcal{C} H_S (\mathbf{x},t- \Delta t/2) , \label{PQdef} \\ 
Q^{\pm} (\mathbf{x},t+\Delta t/2) & = & H_S (\mathbf{x},t+\Delta t/2) \pm S^{\pm} \mathcal{C} E_S (\mathbf{x},t) , \nonumber \\ 
\end{eqnarray}
and using (\ref{Cident}) and (\ref{SpmI}) again we rewrite the evolution formulas (\ref{ES})-(\ref{HS}):
\begin{eqnarray}
P^{+} (\mathbf{x},t+ \Delta t) & = & -S^{-} \mathcal{C} Q^{+} (\mathbf{x}, t+ \Delta t/2) , \nonumber \\ 
P^{-} (\mathbf{x},t+ \Delta t) & = & S^{+} \mathcal{C} Q^{-} (\mathbf{x}, t+ \Delta t/2) , \label{PQev} \\ 
Q^{+} (\mathbf{x},t+ \Delta t/2) & = & S^{-} \mathcal{C} P^{+} (\mathbf{x}, t) , \nonumber \\ 
Q^{-} (\mathbf{x},t+ \Delta t/2) & = & -S^{+} \mathcal{C} P^{-} (\mathbf{x}, t) . \nonumber \\ 
\end{eqnarray}
By the norm preserving properties of the operators $S^{\pm}$ and $\mathcal{C}$ we deduce the basic conservation laws
in any Sobolev norm or seminorm 
\be
\| P^{\pm} (\cdot, t+ \Delta t ) \| = \| Q^{\pm} (\cdot, t+\Delta t/2) \| = \| P^{\pm} (\cdot, t) \| . \label{concon}
\ee

We now note that for polynomial data the recursions (\ref{Vell}), (\ref{Well}) will terminate once the number of spatial derivatives exceeds the
degree. For the tensor-product polynomials of total degree $6m+3$ we are using we have
\bd
V_{\ell}=W_{\ell} = 0, \ \ \ell > 3m+2 .
\ed
Thus if we take $q=3m+2$ the cell polynomials are evolved exactly. Moreover, if we obey the CFL restriction
\be
c \Delta t < \max_k \Delta x_k , \label{CFL}
\ee
then we have the following lemma. Here we define $\tilde{E}^h$, $\tilde{H}^h$ to be the Hermite-Birkhoff interpolants of the vertex data
and define the quantities $P^{\pm,h}$, $Q^{\pm,h}$ as in (\ref{PQdef}). 

\begin{lemma}\label{exactlem}
For the dielectric system, $M=\Gamma_V=\Gamma_W=0$, if $q=3m+2$ and (\ref{CFL}) holds then the
quantities $P^{\pm,h}$, $E_N^h$, $Q^{\pm,h}$ and $H_N^h$ 
computed from the approximations, $\tilde{E}^h$, $\tilde{H}^h$, to the
symmetrized variables satisfy the evolution formulas:
\begin{eqnarray}
P^{+,h} (\mathbf{x},t+ \Delta t) + E_N^h (\mathbf{x},t+ \Delta t) & = & \nonumber \\ - \mathcal{I}_m \left( S^{-} \mathcal{C} Q^{+,h} (\mathbf{x}, t+ \Delta t/2)
-E_N^h (\mathbf{x},t) \right) & & \nonumber \\ -(1 - \mathcal{I}_m ) \left( S^{+} \mathcal{C} Q^{-,h} (\mathbf{x}, t+ \Delta t/2)
+ E_N^h (\mathbf{x},t) \right) , & & \nonumber \\
P^{-,h} (\mathbf{x},t+ \Delta t) + E_N^h (\mathbf{x},t+ \Delta t) & = & \nonumber \\ \mathcal{I}_m \left( S^{+} \mathcal{C} Q^{-,h} (\mathbf{x}, t+ \Delta t/2)
+E_N^h (\mathbf{x},t) \right) & & \nonumber \\ + (1 - \mathcal{I}_m ) \left( S^{-} \mathcal{C} Q^{+,h} (\mathbf{x}, t+ \Delta t/2)
- E_N^h (\mathbf{x},t) \right) , & & \label{PQhev} \\
Q^{+,h} (\mathbf{x},t+ \Delta t/2) + H_N^h (\mathbf{x},t+ \Delta t/2) & = & \nonumber \\
\mathcal{I}_m \left( S^{-} \mathcal{C} P^{+,h} (\mathbf{x}, t)
+H_N^h (\mathbf{x},t + \Delta t/2) \right) & & \nonumber \\ +(1 - \mathcal{I}_m ) \left( S^{+} \mathcal{C} Q^{-,h} (\mathbf{x}, t+ \Delta t/2)
- H_N^h (\mathbf{x},t-\Delta t/2) \right) , & & \nonumber \\
Q^{-,h} (\mathbf{x},t+ \Delta t/2) + H_N^h (\mathbf{x},t+ \Delta t/2) & = & \nonumber \\ - \mathcal{I}_m \left( S^{+} \mathcal{C} P^{-,h} (\mathbf{x}, t)
-H_N^h (\mathbf{x},t- \Delta t/2) \right) & & \nonumber \\ - (1 - \mathcal{I}_m ) \left( S^{-} \mathcal{C} P^{+,h} (\mathbf{x}, t)
+ H_N^h (\mathbf{x},t- \Delta t/2) \right) . & & \nonumber
\end{eqnarray}
\end{lemma}
\begin{proof}
Assuming (\ref{CFL}), the domain of dependence of the solution at the cell centers on either grid lies completely within the cell. Therefore, since
the cell polynomial is updated exactly if we take $q=3m+2$, the data used to compute the Hermite-Birkhoff interpolants is the exact evolution of the
approximate solution at the previous times. Thus the only error over a time step is the interpolation error which can then be projected onto the
various solution components. The further complications in the formulas (\ref{PQhev}) in comparison to (\ref{PQev}) arise from the fact that the projections
do not commute with $\mathcal{I}_m$. Recalling that $\mathcal{I}_m$ is a projection the discrete evolution formulas are
\begin{eqnarray}
\tilde{E}^h (\mathbf{x},t+ \Delta t) & = & \tilde{E}^h (\mathbf{x},t) + \mathcal{I}_m \left( S^{+} - S^{-} \right) H_S^h (\mathbf{x},t+\Delta t/2) ,
\label{Ehplus} \\
\tilde{H}^h (\mathbf{x},t+ \Delta t/2) & = & \tilde{H}^h (\mathbf{x},t-\Delta t/2) - \mathcal{I}_m \left( S^{+} - S^{-} \right) E_S^h (\mathbf{x},t) .
\label{Hhplus}
\end{eqnarray}
Consider, for example, the update formula for $P^{+,h}+E_N^h$ making use of (\ref{Ehplus})-(\ref{Hhplus}) along with (\ref{Cident}) and (\ref{SpmI}).
Note that 
\bd
\tilde{E}^h (\mathbf{x},t)= \mathcal{I}_m \tilde{E}^h (\mathbf{x},t) = - \mathcal{I}_m S^{-} \mathcal{C} S^{+} \mathcal{C} E_S^h(\mathbf{x},t) + \mathcal{I}_m E_N^h(\mathbf{x},t) ,
\ed
\bd
0 = -(1- \mathcal{I}_m ) \tilde{E}^h (\mathbf{x},t) = (1- \mathcal{I}_m ) S^{+}\mathcal{C} S^{-} \mathcal{C} E_S^h(\mathbf{x},t) - (1- \mathcal{I}_m) E_N^h(\mathbf{x},t) .
\ed
We compute
\begin{eqnarray*}
P^{+,h}(\mathbf{x},t+\Delta t) + E_N^h(\mathbf{x},t+\Delta t) & = & \\ \tilde{E}^h(\mathbf{x},t+ \Delta t) - S^{+} \mathcal{C}  H_S^h(\mathbf{x}, t+ \Delta t/2) & = & \\ 
- \mathcal{I}_m \left( S^{-} \mathcal{C} \left( H_S^h(\mathbf{x},t+\Delta t/2)+ S^{+} \mathcal{C} E_S^h(\mathbf{x},t) \right) - E_N^h(\mathbf{x},t) \right) & & \\
- (1 - \mathcal{I}_m ) \left( S^{+} \mathcal{C} \left( H_S^h (\mathbf{x},t+\Delta t/2) - S^{-} \mathcal{C} E_S^h(\mathbf{x},t) \right) + E_N^h (\mathbf{x},t) \right) 
& = & \\ - \mathcal{I}_m \left( S^{-} \mathcal{C} Q^{+,h} (\mathbf{x}, t+ \Delta t/2)
-E_N^h (\mathbf{x},t) \right) & & \\ -(1 - \mathcal{I}_m ) \left( S^{+} \mathcal{C} Q^{-,h} (\mathbf{x}, t+ \Delta t/2)
+ E_N^h (\mathbf{x},t) \right) . & & 
\end{eqnarray*}
The other identities in (\ref{PQhev}) are similarly derived.
\end{proof}

We are now in a position to prove Theorem \ref{HBcon}.

\begin{theorem}\label{HBcon}
For the dielectric system, $M=\Gamma_V=\Gamma_W=0$, if $q=3m+2$ and (\ref{CFL}) holds then
the approximate solution satisfies the conservation laws
\begin{eqnarray}
\left\arrowvert P^{+,h} (\cdot,t+\Delta t) \right\arrowvert_m^2 + \left\arrowvert P^{-,h} (\cdot,t+\Delta t) \right\arrowvert_m^2 & & \nonumber \\ + 2
\left\arrowvert E_N^h (\cdot, t+\Delta t ) \right\arrowvert_m^2 + 2 \left\arrowvert H_N^h (t + \Delta t/2) \right\arrowvert_m^2 & = & \nonumber \\
\left\arrowvert Q^{+,h} (\cdot,t+\Delta t/2) \right\arrowvert_m^2 + \left\arrowvert Q^{-,h} (\cdot,t+\Delta t/2) \right\arrowvert_m^2 & & \label{Pcon} \\  
+ 2 \left\arrowvert H_N^h (\cdot, t+\Delta t/2 ) \right\arrowvert_m^2 + 2 \left\arrowvert E_N^h (t) \right\arrowvert_m^2 & = & \nonumber \\
\left\arrowvert P^{+,h} (\cdot,t) \right\arrowvert_m^2 + \left\arrowvert P^{-,h} (\cdot,t) \right\arrowvert_m^2 + 2
\left\arrowvert E_N^h (\cdot, t ) \right\arrowvert_m^2 + 2 \left\arrowvert H_N^h (t - \Delta t/2) \right\arrowvert_m^2 . & &  \nonumber 
\end{eqnarray}
\end{theorem}
\begin{proof}
Compute recalling the fact that $P^{\pm,h}$ and $Q^{\pm,h}$ are orthogonal to $E^{N,h}$ and $H^{N,h}$ in the HB semi-inner product.
\begin{eqnarray*}
\left\arrowvert P^{+,h} (\cdot,t+\Delta t) \right\arrowvert_m^2 + \left\arrowvert E_N^h (\cdot, t+\Delta t ) \right\arrowvert_m^2 & = & \\
\left\arrowvert P^{+,h} (\cdot,t+\Delta t) +  E_N^h (\cdot, t+\Delta t ) \right\arrowvert_m^2 & = & \\
\left\arrowvert \mathcal{I}_m \left( S^{-} \mathcal{C} Q^{+,h} (\mathbf{x}, t+ \Delta t/2) -E_N^h (\mathbf{x},t) \right) \right\arrowvert_m^2 & & \\ +
\left\arrowvert (1- \mathcal{I}_m ) \left( S^{+} \mathcal{C} Q^{-,h} (\mathbf{x}, t+ \Delta t/2) + E_N^h (\mathbf{x},t) \right) \right\arrowvert_m^2 & = & \\ 
\left\arrowvert \mathcal{I}_m S^{-} \mathcal{C} Q^{+,h} (\mathbf{x}, t+ \Delta t/2) \right\arrowvert_m^2 & & \\ + \left\arrowvert (1-\mathcal{I}_m )
S^{+} \mathcal{C} Q^{-,h} (\mathbf{x}, t+ \Delta t/2) \right\arrowvert_m^2 + \left\arrowvert E_N^h (\mathbf{x},t) \right\arrowvert_m^2 & & \\  
- 2 \langle \mathcal{I}_m S^{-} \mathcal{C} Q^{+,h} (\mathbf{x}, t+ \Delta t/2) , \mathcal{I}_m E_N^h (\mathbf{x},t) \rangle_m & & \\
+ 2 \langle (1-\mathcal{I}_m) S^{+} \mathcal{C} Q^{-,h} (\mathbf{x}, t+ \Delta t/2) , (1- \mathcal{I}_m) E_N^h (\mathbf{x},t) \rangle_m . & & 
\end{eqnarray*}
\begin{eqnarray*}
\left\arrowvert P^{-,h} (\cdot,t+\Delta t) \right\arrowvert_m^2 + \left\arrowvert E_N^h (\cdot, t+\Delta t ) \right\arrowvert_m^2 & = & \\
\left\arrowvert P^{-,h} (\cdot,t+\Delta t) +  E_N^h (\cdot, t+\Delta t ) \right\arrowvert_m^2 & = & \\
\left\arrowvert \mathcal{I}_m \left( S^{+} \mathcal{C} Q^{-,h} (\mathbf{x}, t+ \Delta t/2) +E_N^h (\mathbf{x},t) \right) \right\arrowvert_m^2 & & \\ +
\left\arrowvert (1- \mathcal{I}_m ) \left( S^{-} \mathcal{C} Q^{+,h} (\mathbf{x}, t+ \Delta t/2) - E_N^h (\mathbf{x},t) \right) \right\arrowvert_m^2 & = & \\ 
\left\arrowvert \mathcal{I}_m S^{+} \mathcal{C} Q^{-,h} (\mathbf{x}, t+ \Delta t/2) \right\arrowvert_m^2 & & \\ + \left\arrowvert (1-\mathcal{I}_m )
S^{-} \mathcal{C} Q^{+,h} (\mathbf{x}, t+ \Delta t/2) \right\arrowvert_m^2 + \left\arrowvert E_N^h (\mathbf{x},t) \right\arrowvert_m^2 & & \\  
+ 2 \langle \mathcal{I}_m S^{+} \mathcal{C} Q^{-,h} (\mathbf{x}, t+ \Delta t/2) , \mathcal{I}_m E_N^h (\mathbf{x},t) \rangle_m & & \\
- 2 \langle (1-\mathcal{I}_m) S^{-} \mathcal{C} Q^{+,h} (\mathbf{x}, t+ \Delta t/2) , (1-\mathcal{I}_m) E_N^h (\mathbf{x},t) \rangle_m . & & 
\end{eqnarray*}
Adding these expressions we find
\begin{eqnarray*}
\left\arrowvert P^{+,h} (\cdot,t+\Delta t) \right\arrowvert_m^2 + \left\arrowvert P^{-,h} (\cdot,t+\Delta t) \right\arrowvert_m^2
+2 \left\arrowvert E_N^h (\cdot, t+\Delta t ) \right\arrowvert_m^2 & = & \\ \left\arrowvert S^{-} \mathcal{C} Q^{+,h} (\mathbf{x}, t+ \Delta t/2) \right\arrowvert_m^2
+ \left\arrowvert S^{+} \mathcal{C} Q^{-,h} (\mathbf{x}, t+ \Delta t/2) \right\arrowvert_m^2 + 2 \left\arrowvert E_N^h (\mathbf{x},t) \right\arrowvert_m^2 & & \\
-2 \langle \mathcal{I}_m S^{-} \mathcal{C} Q^{+,h} (\mathbf{x}, t+ \Delta t/2) , \mathcal{I}_m E_N^h (\mathbf{x},t) \rangle_m & & \\
-2 \langle (1-\mathcal{I}_m) S^{-} \mathcal{C} Q^{+,h} (\mathbf{x}, t+ \Delta t/2) , \mathcal{I}_m E_N^h (\mathbf{x},t) \rangle_m & & \\
+ 2 \langle (1-\mathcal{I}_m) S^{+} \mathcal{C} Q^{-,h} (\mathbf{x}, t+ \Delta t/2) , (1- \mathcal{I}_m) E_N^h (\mathbf{x},t) \rangle_m & & \\
+ 2 \langle \mathcal{I}_m S^{+} \mathcal{C} Q^{-,h} (\mathbf{x}, t+ \Delta t/2) , \mathcal{I}_m E_N^h (\mathbf{x},t) \rangle_m & = & \\
\left\arrowvert Q^{+,h} (\mathbf{x}, t+ \Delta t/2) \right\arrowvert_m^2 + \left\arrowvert Q^{-,h} (\mathbf{x}, t+ \Delta t/2) \right\arrowvert_m^2
+2 \left\arrowvert E_N^h (\mathbf{x},t) \right\arrowvert_m^2 & & \\
-2 \langle Q^{+,h} (\mathbf{x}, t+ \Delta t/2) , E_N^h (\mathbf{x},t) \rangle_m & & \\
+ 2 \langle Q^{-,h} (\mathbf{x}, t+ \Delta t/2) , E_N^h (\mathbf{x},t) \rangle_m & = & \\
\left\arrowvert Q^{+,h} (\mathbf{x}, t+ \Delta t/2) \right\arrowvert_m^2 + \left\arrowvert Q^{-,h} (\mathbf{x}, t+ \Delta t/2) \right\arrowvert_m^2
+2 \left\arrowvert E_N^h (\mathbf{x},t) \right\arrowvert_m^2 & & .
\end{eqnarray*}
Adding $2 \left\arrowvert H_N^h (t + \Delta t/2) \right\arrowvert_m^2$ to both sides of this equation yields the first equality in (\ref{Pcon}).
The second is proven similarly using the update formulas for $Q^{\pm,h}$ in (\ref{PQhev}).
\end{proof}

Having established stability we move on to derive error estimates. To this end it
is useful to organize the previous results in terms of the evolution of the conserved quantities. Specifically we introduce the
HB-seminorm conserving operators $\mathcal{U}$ and $\mathcal{V}$:
\be
\left( \ba{c} P^{+,h}(\mathbf{x}, t+\Delta t) \\ P^{-,h}(\mathbf{x}, t+\Delta t) \\ E_N^h (\mathbf{x}, t+\Delta t) \\ H_N^h (\mathbf{x}, t+\Delta t/2) \ea \right) = \mathcal{U}^h 
\left( \ba{c} Q^{+,h}(\mathbf{x}, t+\Delta t/2) \\ Q^{-,h}(\mathbf{x}, t+\Delta t/2) \\ H_N^h (\mathbf{x}, t+\Delta t/2) \\ E_N^h (\mathbf{x}, t) \ea \right) , \label{calUdef}
\ee
\be
\left( \ba{c} Q^{+,h}(\mathbf{x}, t+\Delta t/2) \\ Q^{-,h}(\mathbf{x}, t+\Delta t/2) \\ H_N^h (\mathbf{x}, t+\Delta t/2) \\ E_N^h (\mathbf{x}, t) \ea \right) = \mathcal{V}^h 
\left( \ba{c} P^{+,h}(\mathbf{x}, t) \\ P^{-,h}(\mathbf{x}, t) \\ E_N^h (\mathbf{x}, t) \\ H_N^h (\mathbf{x}, t-\Delta t/2) \ea \right) . \label{calVdef}
\ee

Define errors in the conserved quantities by
\be
\mathcal{E}(t_n) = \left( \ba{c} P^{+}(\mathbf{x}, t_n) - P^{+,h}(\mathbf{x}, t_n) \\ P^{-}(\mathbf{x}, t_n) - P^{-,h}(\mathbf{x}, t_n)  \\
E_N(\mathbf{x}, t_n) - E_N^h (\mathbf{x}, t_n) \\ H_N(\mathbf{x} , t_{n-1/2}) - H_N^h (\mathbf{x} , t_{n-1/2}) \ea \right) , \label{errndef}
\ee
\be
\mathcal{E}(t_{n+1/2}) = \left( \ba{c} Q^{+}(\mathbf{x}, t_{n+1/2}) - Q^{+,h}(\mathbf{x}, t_{n+1/2}) \\ Q^{-}(\mathbf{x}, t_{n+1/2}) - P^{-,h}(\mathbf{x}, t_{n+1/2})  \\
H_N(\mathbf{x}, t_{n+1/2}) - H_N^h (\mathbf{x}, t_{n+1/2}) \\ E_N(\mathbf{x} , t_{n}) - E_N^h (\mathbf{x} , t_{n}) \ea \right) , \label{errhdef}
\ee
where we have introduced $t_n=n \Delta t$, $t_{n+1/2} = (n+1/2)\Delta t$. Convergence in the HB seminorm is established in Theorem \ref{HBerr}

\begin{theorem}\label{HBerr}
For the dielectric system, $M=\Gamma_V=\Gamma_W=0$, if $q=3m+2$, the CFL number $\f {\Delta t}{h}$ is fixed and satisfies (\ref{CFL}), and the initial approximations
are sufficiently accurate, there exists $C$ depending only on $m$, the CFL number,
and derivatives of the solution $\tilde{E}(\mathbf{x},t)$, $\tilde{H}(\mathbf{x},t)$, such that, for $h=\max (h_x,h_y,h_z)$
\be
\left\arrowvert \mathcal{E}(t_n) \right\arrowvert_m + \left\arrowvert \mathcal{E}(t_{n+1/2}) \right\arrowvert_m \leq C \left(1 + t_n \right) h^{m+1} . \label{errest}
\ee
\end{theorem}
\begin{proof}
We combine (\ref{PQhev}) with the exact formulas (\ref{ES})-(\ref{HN}) to derive evolution formulas for $\mathcal{E}$. We begin with the first two equations.
Note that we are assuming $E_N=0$ and $H_N=0$ but for clarity we retain them in the error equations. 
\begin{eqnarray*}
P^{+}(\mathbf{x}, t_{n+1}) - P^{+,h}(\mathbf{x}, t_{n+1}) + E_N (\mathbf{x},t_{n+1}) - E_N^h (\mathbf{x},t_{n+1}) & = & \\
- \mathcal{I}_m \left( S^{-} \mathcal{C} \left(Q^{+}(\mathbf{x}, t_{n+1/2}) -  Q^{+,h} (\mathbf{x}, t_{n+1/2}) \right) - \left( E_N (\mathbf{x},t_{n+1})  
-E_N^h (\mathbf{x},t_n) \right) \right) & & \\ -(1 - \mathcal{I}_m ) \left( S^{+} \mathcal{C}
\left(Q^{-}(\mathbf{x}, t_{n+1/2}) -  Q^{-,h} (\mathbf{x}, t_{n+1/2}) \right)  
+ \left( E_N (\mathbf{x},t_{n}) - E_N^h (\mathbf{x},t_n) \right) \right) & & \\ 
+ (1 - \mathcal{I}_m ) \left( E_S (\mathbf{x},t_{n+1}) - E_S( \mathbf{x},t_n ) \right) , & & 
\end{eqnarray*}
\begin{eqnarray*}
P^{-}(\mathbf{x}, t_{n+1}) - P^{-,h}(\mathbf{x}, t_{n}) + E_N (\mathbf{x},t_{n+1}) - E_N^h (\mathbf{x},t_{n+1}) & = & \\
+ \mathcal{I}_m \left( S^{+} \mathcal{C} \left(Q^{-}(\mathbf{x}, t_{n+1/2}) -  Q^{-,h} (\mathbf{x}, t_{n+1/2}) \right) + \left( E_N (\mathbf{x},t_{n+1})  
-E_N^h (\mathbf{x},t_n) \right) \right) & & \\  + (1 - \mathcal{I}_m ) \left( S^{-} \mathcal{C}
\left(Q^{+}(\mathbf{x}, t_{n+1/2}) -  Q^{+,h} (\mathbf{x}, t_{n+1/2}) \right)  
- \left( E_N (\mathbf{x},t_{n}) - E_N^h (\mathbf{x},t_n) \right) \right) & & \\  + (1 - \mathcal{I}_m ) \left( E_S (\mathbf{x},t_{n+1}) - E_S( \mathbf{x},t_n ) \right) . & & 
\end{eqnarray*}
Similarly we can write down the evolution of the error for $E_N$
\begin{eqnarray*}
E_N (\mathbf{x},t_{n+1}) - E_N^h (\mathbf{x},t_{n+1}) & = & N \mathcal{I}_m \left( S^{+} - S^{-} \right) \mathcal{C} \left( H_S (\mathbf{x},t_{n+1/2}) - H_S^h (\mathbf{x},t_{n+1/2}) \right)
\\ & & + N (1- \mathcal{I}_m) \left( E_S (\mathbf{x},t_{n+1}) - E_S( \mathbf{x},t_n ) \right)  .
\end{eqnarray*} 
Note that the $H_N$ error is simply copied in this step of the evolution. 
We can rewrite these relations as
\bd
\mathcal{E}(t_{n+1}) = \mathcal{U}^h \mathcal{E}(t_{n+1/2}) + \tau_{n+1}
\ed
where
\bd
\tau_{n+1} = \left( \ba{c} \left[ (1 - \mathcal{I}_m ) \left( E_S (\mathbf{x},t_{n+1}) - E_S( \mathbf{x},t_n ) \right) \right]_S \\
\left[ (1 - \mathcal{I}_m ) \left( E_S (\mathbf{x},t_{n+1}) - E_S( \mathbf{x},t_n ) \right) \right]_S \\
\left[ (1 - \mathcal{I}_m ) \left( E_S (\mathbf{x},t_{n+1}) - E_S( \mathbf{x},t_n ) \right) \right]_N \\ 0 \ea \right) .
\ed
Assuming a smooth solution to the continuous problem the standard Hermite-Birkhoff interpolation error formulas (e.g \ci{GoodHer}) imply:
\bd
\left\arrowvert (1 - \mathcal{I}_m ) \left( E_S (\mathbf{x},t_{n+1}) - E_S( \mathbf{x},t_n ) \right) \right\arrowvert_m \leq C \f {\Delta t}{2} \ h^{m+1} .
\ed
So invoking the triangle inequality for the seminorm and the fact that $\mathcal{U}^h$ preserves the seminorm we have
\bd
\left\arrowvert \mathcal{E}(t_{n+1}) \right\arrowvert_m \leq \left\arrowvert \mathcal{U}^h \mathcal{E}(t_{n+1/2}) \right\arrowvert_m +
\left\arrowvert \tau_{n+1} \right\arrowvert_m \leq \left\arrowvert \mathcal{E}(t_{n+1/2}) \right\arrowvert_m + C \f {\Delta t}{2} \ h^{m+1} .
\ed
By similar computations we deduce 
\bd
\left\arrowvert \mathcal{E}(t_{n+1/2}) \right\arrowvert_m \leq \left\arrowvert \mathcal{E}(t_{n}) \right\arrowvert_m + C \f {\Delta t}{2} \ h^{m+1} .
\ed
Summing these inequalities we have
\bd
\left\arrowvert \mathcal{E}(t_{n}) \right\arrowvert_m + \left\arrowvert \mathcal{E}(t_{n+1/2}) \right\arrowvert_m  \leq
\left\arrowvert \mathcal{E}(t_{0}) \right\arrowvert_m + \left\arrowvert \mathcal{E}(t_{1/2}) \right\arrowvert_m +C t_n h^{m+1} .
\ed
Assuming, as would be the case if we intepolates a smooth initial condition and initial half-step,
\bd
\left\arrowvert \mathcal{E}(t_{0}) \right\arrowvert_m \leq Ch^{m+1}, \ \ \left\arrowvert \mathcal{E}(t_{1/2}) \right\arrowvert_m \leq Ch^{m+1} ,
\ed
we obtain the final result. 

\end{proof}

\subsection{Extensions to the Dispersive System}

Inclusion of the dispersive terms does not change the domain-of-dependence of the exact solution.
Thus, excluding the dissipative terms, we believe that the previous analysis could be repeated
via the definition of the complex exponentials of the operators appearing in the update formulas.
However, we can no longer expect to evolve the
cell polynomials exactly and so need to take account of additional sources of error.
Therefore we will follow the standard analysis of stability for leap-frog schemes as presented in
\ci{JolyRev}. Again ignoring the dissipative term and using the fact that $\mathcal{I}_m$ is a projection we can write the
discrete evolution equations (\ref{Vup}), (\ref{Wup}), in the form:
\begin{eqnarray*}
V(\mathbf{x},t_{n+1}) - V(\mathbf{x},t_n) & = & \mathcal{I}_m \mathcal{B} \mathcal{I}_m W(\mathbf{x},t_{n+1/2}) , \\
W(\mathbf{x},t_{n+1/2}) - W(\mathbf{x},t_{n-1/2}) & = & - \mathcal{I}_m \mathcal{B}^{\ast} \mathcal{I}_m V(\mathbf{x},t_{n}) ,
\end{eqnarray*}
where 
\be
\mathcal{B} \mathcal{I}_m W =\sum_{\ell=1}^q \f {\left( \Delta t/2 \right)^{2 \ell -1}}
{(2 \ell -1 )!} V^{\ell} , \label{Vdef}
\ee
with $V^{\ell}$ defined by (\ref{V1})-(\ref{Vell}).
Now consider the scaling of terms in (\ref{V1})-(\ref{Vell}). When restricted to the polynomial space the derivative operators
$\f {\pa}{\pa x_k} \propto \Delta x_k^{-1}$. Therefore, for fixed CFL numbers, $\lambda_k = c \Delta t/\Delta x_k$, we have
\be
\mathcal{I}_m \mathcal{B} \mathcal{I}_m = \mathcal{I}_m \left( \mathcal{B}_0 + \Delta t \mathcal{B}_1 \right) \mathcal{I}_m , \label{Bpert}
\ee
where $\mathcal{B}_0$ is independent of $\Delta t$ and $\mathcal{B}_1$ is bounded. In particular $\mathcal{B}_0$ is the evolution operator
for the dielectric case combined with some additional zero blocks. 

We then have Theorem \ref{HBGenCon} and conditional stability follows.

\begin{theorem}\label{HBGenCon}
For the conservative system, $\Gamma_V=\Gamma_W=0$, the following quantities are constant:
\be
\left\arrowvert V(\cdot,t_{n}) \right\arrowvert_m^2 + \f {1}{4} \left\arrowvert W(\cdot, t_{n+1/2}) + W(\cdot, t_{n-1/2}) \right\arrowvert_m^2 - \f {1}{4} \left\arrowvert
\mathcal{I}_m B^{\ast} V(\cdot, t_n) \right\arrowvert_m^2 ,  \label{Egenn}
\ee
\be
\left\arrowvert W(\cdot,t_{n+1/2}) \right\arrowvert_m^2 + \f {1}{4} \left\arrowvert V(\cdot, t_{n+1}) + V(\cdot, t_{n}) \right\arrowvert_m^2 - \f {1}{4} \left\arrowvert
\mathcal{I}_m B^{\ast} W(\cdot, t_{n+1/2}) \right\arrowvert_m^2 .  \label{Egenh}
\ee
\end{theorem}
\begin{corollary}\label{HBGenStab}
For the conservative system, $\Gamma_V=\Gamma_W=0$, we have stability in the HB seminorm if
\be
\| \mathcal{B} \|_m < 4, \ \ \ \ \| \mathcal{B}^{\ast} \|_m < 4 .  \label{GenStab}
\ee
\end{corollary}
\begin{remark}
Condition (\ref{GenStab}) plays the role of a CFL condition. Invoking (\ref{Bpert}) and Theorem \ref{HBcon} and taking $q = 3m+2$ we expect stability
for $\Delta t$ small enough under the domain-of-dependence
CFL condition (\ref{CFL}). In our experiments we find that the method is stable in both the dielectric and dispersive case with large CFL numbers and smaller
values of $q$. 
\end{remark}
\begin{proof}
Combining two steps in (\ref{Wup}) we have the formula
\bd
W(\mathbf{x},t_{n+3/2}) - W(\mathbf{x},t_{n-1/2}) = - \mathcal{I}_m \mathcal{B}^{\ast} \mathcal{I}_m \left( V(\mathbf{x},t_{n+1}) + V(\mathbf{x},t_{n}) \right) .
\ed
Taking the HB inner product of this equation with $W(\mathbf{x},t_{n+3/2})$ and the HB inner product of (\ref{Vup}) with $V(\mathbf{x},t_{n+1}) + V(\mathbf{x},t_{n})$ we obtain
\begin{eqnarray*}
\langle W(\cdot ,t_{n+3/2}) , W(\cdot ,t_{n+1/2}) \rangle_m & = & \langle W(\cdot ,t_{n+1/2}) , W(\cdot ,t_{n-1/2}) \rangle_m \\ & & - \langle W(\cdot ,t_{n+1/2}),     
\mathcal{I}_m \mathcal{B}^{\ast} \mathcal{I}_m \left( V(\cdot,t_{n+1}) + V(\cdot,t_{n}) \right) \rangle_m ,
\end{eqnarray*}
\bd
\left\arrowvert V(\cdot ,t_{n+1}) \right\arrowvert_m^2  =  \left\arrowvert V(\cdot ,t_{n}) \right\arrowvert_m^2 
+ \langle \mathcal{I}_m \mathcal{B} \mathcal{I}_m W(\cdot ,t_{n+1/2}), V(\cdot,t_{n+1}) + V(\cdot,t_{n}) \rangle_m .
\ed
Adding these expressions and noting that the terms involving $\mathcal{B}$ and $\mathcal{B}^{\ast}$ cancel we deduce that the quantity
\bd
\left\arrowvert V(\cdot ,t_{n}) \right\arrowvert_m^2 + \langle W(\cdot ,t_{n+1/2}) , W(\cdot ,t_{n-1/2}) \rangle_m
\ed
is constant. We rewrite the second term by noting that
\begin{eqnarray*}
\f {1}{4} \left\arrowvert W(\cdot, t_{n+1/2}) + W(\cdot, t_{n-1/2}) \right\arrowvert_m^2 - \f {1}{4} \left\arrowvert W(\cdot, t_{n+1/2}) - W(\cdot, t_{n-1/2}) \right\arrowvert_m^2 
& = & \\ \langle W(\cdot ,t_{n+1/2}) , W(\cdot ,t_{n-1/2}) \rangle_m . & &
\end{eqnarray*} 
Replacing the difference term with (\ref{Wup}) yields (\ref{Egenn}). Equation (\ref{Egenh}) is derived by the analogous procedure.
\end{proof}

Given the stability Theorem, error estimates in the HB seminorm can also be obtained by standard means. We will not present them here, but instead focus on
observing stability bounds and convergence rates in $L^2$ in numerical experiments. One can attempt a standard analysis of convergence by studying the local
truncation error method and the associated stability of the scheme. The approximation of derivatives by the Hermite interpolant of a smooth function
at the cell centers will have errors which scale with $\arrowvert \Delta x \arrowvert^{2m+2-j}$ for derivatives of even order and
$\arrowvert \Delta x \arrowvert^{2m+3-j}$ for derivatives of odd order.
Therefore, if we consider the evolution of the scaled discrete data (\ref{Vhdef})-(\ref{Whdef}) and take $q \geq m$ we derive estimates of the local truncation
error of order $\Delta t \arrowvert \Delta x \arrowvert^{2m+2}$ for even derivatives and $\Delta t \arrowvert \Delta x \arrowvert^{2m+1}$ for odd derivatives.
However, translating stability from the HB-seminorm to $L^2$ is not straightforward, and, as discussed in \ci{HermiteLF}, one can at best expect convergence
at order $2m$. Moreover, the energy and $L^2$ error bounds derived in \ci{HermiteLF} degrade in time by factors of $t^2$ and $t^3$ respectively; we have never observed this growth in numerical experiments. However, it is observed in \ci{HermiteLF} that for $m$ even the conservative approximation to the acoustic system leads to convergence at
order $2m+2$ and an argument is presented in one space dimension to explain it. The focus here is on methods with $m \geq 3$ and our experiments do not
unambiguously determine if this phenomenon occurs, though least squares fits to the convergence rate for $m=4$ do generally exceed $10$. One experiment
with $m=2$ does very clearly exhibit convergence at order $6$, and so we conjecture that the convergence rate is in fact $2m+2$ for $m$ even. 

\section{Numerical Experiments}\label{sec:num}

Here we present some illustrations of the performance of the proposed methods for the transverse magnetic reduction of Maxwell's equations in
two space dimensions. In all our examples we simply impose periodicity in space. We note that for these simple domains
it is also straightforward to implement perfect electric conductor boundary conditions by imposing appropriate even and odd extensions of the electric
fields at the boundary. This would be a primitive version of the compatability boundary condition method mentioned above.

\subsection{Dielectric Medium}\label{sec:dielectric}

We evolve solutions of the form:
\begin{eqnarray*}
E_x & = & - \f {k}{\epsilon \omega} \sin{(kx)} \cos{(ky)} \cos{(\omega t)} \\
E_y & = &  \f {k}{\epsilon \omega} \cos{(kx)} \sin{(ky)} \cos{(\omega t)} \\
H_z & = &  \sin{(kx)} \sin{(ky)} \sin{(\omega t)} ,
\end{eqnarray*}
with $(x,y) \in (-\pi , \pi)^2$, $\epsilon = \f {5}{4}$, $\mu = \f {4}{5}$, $\omega=\sqrt{2} k$. We present two sets of experiments. The first is simply to
examine the convergence rates for various values of $m$. The second is to
compare efficiency in terms of total degrees-of-freedom required as well as
CPU time for problems of varying difficulty and error tolerances.

\subsubsection*{Convergence}

We fix $k=40$ and solve to $T=100$. Since the wave speed is $1$ and the
wavelength is $2 \pi/40$ a wave can travel approximately $636.6$ wavelengths during the simulation
and in time there are $900.3$ periods. Here we vary $m$ from $3$ to $6$ and sample with mesh sizes in
convenient increments starting with meshes which produce errors roughly from $1\%$ to $10\%$.
Precisely for $m=3$ we take $\Delta_x=\Delta_y=2 \pi/N_G$ with $N_G=125:25:275$. For $m=4$ we
take $N_G=100:25:250$, for $m=5$ $N_G=50:25:200$, and for $m=6$ $N_G=40:20:160$. In our comparisons
we will always consider degrees-of-freedom per wavelength for each coordinate
direction. As the number of degrees-of-freedom is $m+1$ the ranges here begin with
$12.5$ for $m=3,4$, $7.5$ for $m=5$, and $7.0$ for $m=6$. 
In all cases we choose
\bd
   {\rm CFL} = \f {\Delta t}{\Delta_x} = 0.9 ,
\ed
and use a temporal order $q=m+2$. Recall that our stability proofs assume a much larger temporal order,
$q=3m+2$, but our experiments show that $q=m+2$ is sufficient for $m=3-6$ and that increasing $q$ does
not improve accuracy. In some experiments below we use even higher values of $m$ where we found that
$q$ needed to be $m+3$ for stability with ${\rm CFL}=0.9$. 

We approximate the
$L^2$ error every time step by interpolating the Hermite solution onto
a $2m \times 2m$ mesh in each cell and summing the results. Precisely we define the relative error
\be
E^2(m,\Delta_x) = ( 4 \sqrt{N_T})^{-1} \sum_{j=1}^{N_T} \| H_z^{\Delta_x} (t_j) - H_z(t_j) \|^2_{\ell^2} , \label{Edef}
\ee
where $N_T$ is the number of steps and $H_z^{\Delta x}$ represents the interpolant of the Hermite solution onto
the fine mesh. (Here we have used $\pi$ as the relative scale since it is the maximum $L^2$-norm of $H_z$.)
The results, shown in Figure \ref{convfig-dielectric}, are sometimes choppy from mesh to mesh. We
display a linear least squares fit to the log of the error as a function of the log of ${\rm DOF}/\lambda$. This produces
convergence rate estimates shown in Table \ref{convtab-dielectric}. In all cases these meet or excede the theoretical
rate of $2m$. For the cases $m=3-4$ they are in fact consistent with $2m+2$, but we do not claim that these
methods have a theoretical convergence rate greater than $2m$.

\begin{figure}[htb]
\begin{center}
\includegraphics[width=0.99\textwidth,trim={0.65cm 0.0cm 1.0cm 0.2cm},clip]{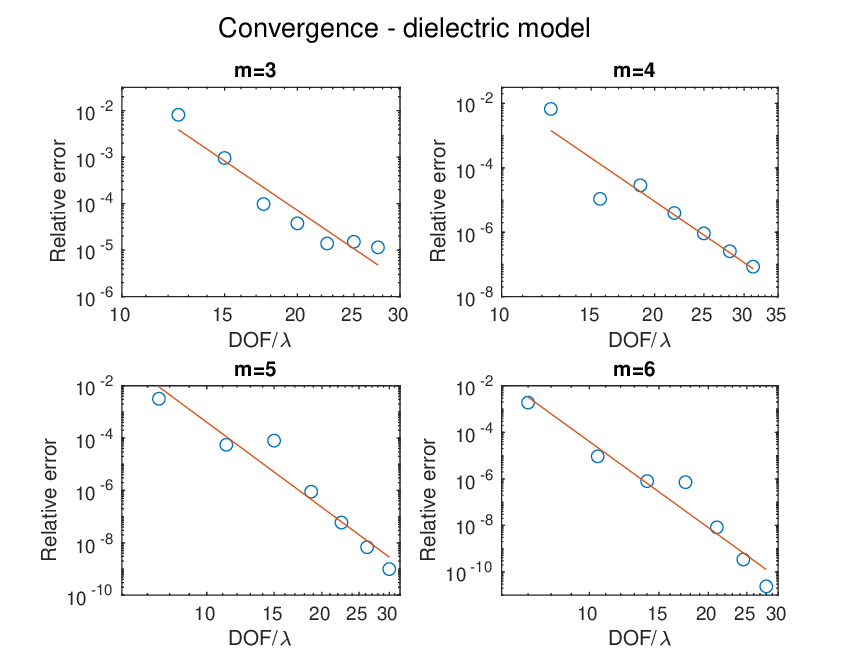}
\caption{Convergence for $m=3-6$, $k=40$, $T=100$ in a dielectric medium. Errors
are computed by comparing the Hermite interpolant of the
numerical solution to the exact solution for each time step.\label{convfig-dielectric}}
\end{center}
\end{figure}

\begin{table}[htb]
\adjustbox{}{
\begin{tabular}{|c|c|r|} \hline
$m$ & DOF/$\lambda$ & Fit Rate \\ \hline \hline
$3$ & $12.5-27.5$ & $8.5$ \\ \hline
$4$ & $12.5-31.25$ & $10.8$  \\ \hline
$5$ & $7.5-30$ & $10.8$ \\ \hline 
$6$ & $7-28$ & $12.3$ \\ \hline \hline
\end{tabular}
}
\caption{Observed convergence for $m=3-6$ with $k=40$ and $T=100$ for the dielectric medium. Here DOF/$\lambda$ denotes
the number of degrees-of-freedom per wavelength in each coordinate direction, 
$N_G(m+1)/k$ where the mesh is $N_G \times N_G$. That is $\Delta_x=\Delta_y = 2 \pi/N_G$. The error is
computed by (\ref{Edef}).\label{convtab-dielectric}}
\end{table} 

The convergence arising from the various choices for $m$ is directly compared in Figure \ref{compfig-dielectric}.
The results in general show that for any particular error level the larger values of $m$
are more efficient in terms of degrees-of-freedom required. We will further examine the
efficiency question in detail
for more challenging problems below. 

\begin{figure}[htb]
\begin{center}
\includegraphics[width=0.99\textwidth,trim={0.4cm 0.0cm 1.0cm 0.2cm},clip]{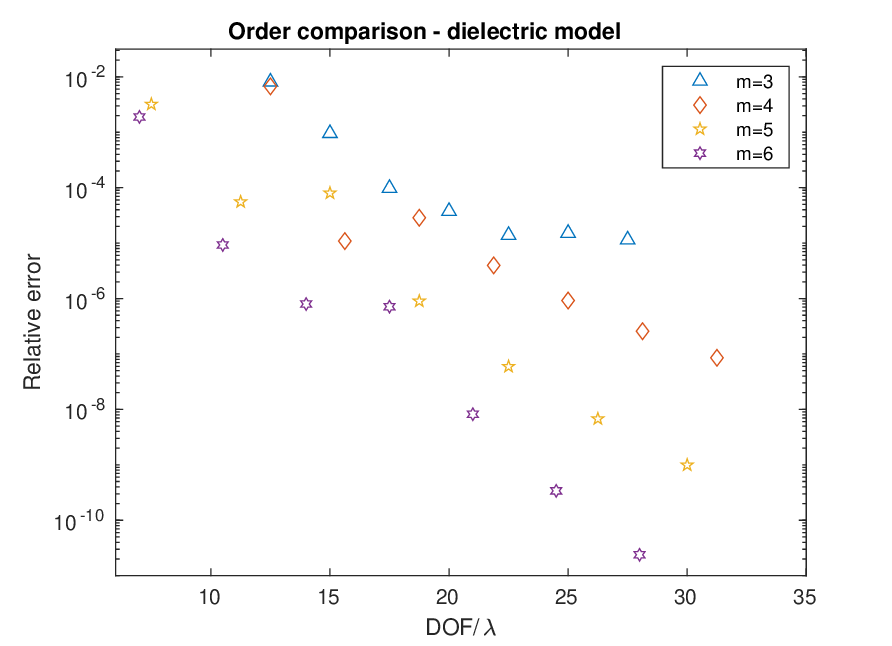}
\caption{Comparison of accuracy for $m=3-6$, $k=40$, $T=100$ in the dielectric medium. Errors
are computed by comparing the Hermite interpolant of the
numerical solution to the exact solution for each time step.\label{compfig-dielectric}}
\end{center}
\end{figure}

\subsubsection*{Efficiency}

Here we consider larger values of $k$, $k=50,100,200$, still evolving to $T=100$.
Varying $\Delta_x=\Delta_y = \f {2 \pi}{N_G}$ by sampling $N_G$ in increments of $5$ we determine the
coarsest mesh for which the maximum recorded error relative to the maximum of
the $L^2$-norm of $H_z$ ($\pi$) is below the tolerances
$\tau = 1 \%$ and $\tau = 0.1 \%$.
Note that for these more challenging experiments waves propagate for approximately $796$,
$1592$ and $3183$ wavelengths,
respectively, corresponding to approximately $1125$, $2251$, and $4502$. temporal periods. As the problem difficulty increased
we also increased the values of $m$ tested to $m=7$ and $m=8$
as the lower order schemes became clearly
less competitive. For these cases we needed to set $q=m+3$ to maintain stability at ${\rm CFL}=0.9$.
For comparison we also tabulate the CPU times in seconds, which are obviously
dependent on the implementation and hardware\footnote{We implemented the method in
Fortran 90, compiling with {\em gfortran} and an optimization -O4. The hardware is
a single Intel i9-13900H core with 64 GiB memory.}, and vary somewhat with repeated runs.
Nonetheless we think the comparisons are still of interest. We remark that for
the larger values of $m$ the cell widths themselves were larger than a wavelength. 

\begin{table}[htb]
\adjustbox{}{
\begin{tabular}{|c|r|c|r|r|c|c|} \hline
$m$ & $k$ & $\tau$ & $L$ & DOF/$\lambda$ & Time & $E_{\rm max}$ \\ \hline \hline
$3$ & $50$ & $1e(-2)$ & $200$ & $16.0$ & $7.8(2)$ & $8.1(-3)$ \\ \hline
$4$ & $50$ & $1e(-2)$ & $140$ & $14.0$ & $5.5(2)$ & $8.5(-3)$ \\ \hline
$5$ & $50$ & $1e(-2)$ & $70$ & $8.4$ & $1.3(2)$ & $6.0(-3)$  \\ \hline
$6$ & $50$ & $1e(-2)$ & $55$ & $7.7$ & $8.8(1)$ & $2.5(-3)$   \\ \hline \hline
$3$ & $50$ & $1e(-3)$ & $250$ & $20.0$ & $2.9(3)$ & $7.3(-4)$  \\ \hline
$4$ & $50$ & $1e(-3)$ & $185$ & $18.5$ & $1.1(3)$ & $6.8(-4)$  \\ \hline
$5$ & $50$ & $1e(-3)$ & $85$ & $10.2$ & $2.1(2)$ & $3.0(-4)$  \\ \hline
$6$ & $50$ & $1e(-3)$ & $65$ & $9.1$ & $1.5(2)$ & $3.8(-4)$   \\ \hline \hline 
$5$ & $100$ & $1e(-2)$ & $160$ & $9.6$ & $1.4(3)$ & $1.5(-2)$ \\ \hline
$6$ & $100$ & $1e(-2)$ & $120$ & $8.4$ & $9.3(2)$ & $9.7(-3)$  \\ \hline
$7$ & $100$ & $1e(-2)$ & $100$  & $8.0$ & $9.5(2)$ & $3.9(-3)$  \\ \hline \hline
$5$ & $100$ & $1e(-3)$ & $170$ & $10.2$ & $1.7(3)$ & $6.6(-4)$  \\ \hline
$6$ & $100$ & $1e(-3)$ & $130$ & $9.1$ & $1.2(3)$ & $7.9(-4)$ \\ \hline
$7$ & $100$ & $1e(-3)$ & $115$ & $9.2$ & $1.2(3)$ & $3.8(-4)$ \\ \hline \hline 
$6$ & $200$ & $1e(-2)$ & $250$ & $8.8$ & $9.0(3)$ & $6.2(-3)$ \\ \hline
$7$ & $200$ & $1e(-2)$ & $200$ & $7.9$ & $8.4(3)$ & $7.9(-3)$ \\ \hline
$8$ & $200$ & $1e(-2)$ & $180$ & $8.1$ & $1.7(4)$ & $5.5(-3)$ \\ \hline \hline
$6$ & $200$ & $1e(-3)$ & $265$ & $9.3$ & $9.5(3)$ & $2.8(-4)$ \\ \hline
$7$ & $200$ & $1e(-3)$ & $230$ & $9.2$ & $1.2(4)$ & $7.5(-4)$ \\ \hline
$8$ & $200$ & $1e(-3)$ & $195$ & $8.8$ & $1.8(4)$ & $5.3(-4)$ \\ \hline \hline 
\end{tabular}
}
\caption{Values of $N$ required for various values of $m$ to achieve tolerances of
$1\%$ and $0.1\%$ at $T=100$ for various values of $k$.\label{mcomptab}}
\end{table} 

The results, shown in Table \ref{mcomptab}, clearly illustrate the effectiveness of
the high order methods and the small computational overhead for Hermite methods
as $m$ is increased. In all cases we achieve the desired tolerances with around
$8-9$ degrees of freedom per wavelength if we choose $m$ large enough. Moreover,
due to the the fact that we can choose the CFL number independent of $m$,
the large $m$ runs were often faster. Although we are cognizant of the pitfalls in
the interpretation of timing data, we still believe that it is worth noting that
in all but one case $m=6$ achieved the tolerances in the least measured CPU time,
and the increase in the number of degrees of freedom per wavelength is quite mild;
for example for the $0.1\%$ tolerance the $m=6$ runs required $9.1$, $9.1$, and $9.3$
as $k$ was increased from $50$ to $200$.

\subsection{Dispersive Medium} 

As an example we consider a single-term Lorentz model for the permittivity and 
approximate solutions of the form considered in \ci{BokilGibson14}. Specifically
we take:
\bd
\epsilon=\mu=1, \ \ \Omega_{e,1}=1, \ \ \omega_{e,1}=\sqrt{1.052 \pi} \approx 1.818, \ \ \gamma_{e,1}=.0107 ,
\ed
which can be obtained by scaling the model for cubic silicon carbide listed in \ci{OpticsHandbook}. We will
also carry out experiments for the Sellmeier model obtained by setting $\gamma_{e,1}=0$. 

Assuming $2 \pi$-periodicity we again take $k=40$ and $T=100$ and approximate solutions of the form
\begin{eqnarray*}
E_x & = & - \f {1}{2k} \sin{(kx)} \cos{(ky)} \left(\omega \cos{(\omega t)} - \theta \sin{(\omega t)} \right) e^{- \theta t} , \\
E_y & = &  \f {1}{2k} \cos{(kx)} \sin{(ky)} \left(\omega \cos{(\omega t)} - \theta \sin{(\omega t)} \right) e^{- \theta t} , \\ 
H_z & = &  \sin{(kx)} \sin{(ky)} \sin{(\omega t)} e^{-\theta t} , \\
K_x & = & \f {1}{2k \omega_{e,1}^2} \sin{(kx)} \cos{(ky)} \left( -2 \omega \theta \cos{(\omega t)}
+ \left(2k^2+\theta^2 - \omega^2 \right) \sin{(\omega t)} \right) e^{- \theta t} , \\ 
K_y & = & -\f {1}{2k \omega_{e,1}^2} \cos{(kx)} \sin{(ky)} \left( -2 \omega \theta \cos{(\omega t)}
+ \left(2k^2+\theta^2 - \omega^2 \right) \sin{(\omega t)} \right) e^{- \theta t} , \\ 
L_x & = & \f {1}{2k \omega_{e,1}^2 \left( \theta^2 + \omega^2 \right)} \sin{(kx)} \cos{(ky)} \\ & & \times\left(
\left( \theta^2 + \omega^2 -2k^2 \right) \omega \cos{(\omega t)} -
\left(2k^2+\theta^2 + \omega^2 \right) \theta \sin{(\omega t)} \right) e^{- \theta t} , \\ 
L_y & = & -\f {1}{2k \omega_{e,1}^2 \left( \theta^2 + \omega^2 \right)} \cos{(kx)} \sin{(ky)} \\ & & \times \left(
\left( \theta^2 + \omega^2 -2k^2 \right) \omega \cos{(\omega t)} -
\left(2k^2+\theta^2 + \omega^2 \right) \theta \sin{(\omega t)} \right) e^{- \theta t} ,
\end{eqnarray*}
where $z=-\theta + i \omega$ is a root of the quartic equation
\bd
z^4 + \gamma_{e,1} z^3 + \left( 2k^2 + \omega_{e,1}^2 + \Omega_{e,1}^2 \right) z^2 + 2k^2 \gamma_{e,1} z + 2k^2 \Omega_{e,1}^2 =0.
\ed
For our choice of parameters we compute the roots (labelled $r$ for resonant and $h$ for high-frequency):
\begin{eqnarray}
\theta_r = 0.005344476784229 & &  \omega_r = 0.999469550181686 , \label{resroot} \\
\theta_h = 0.000005523215771 & &  \omega_h = 56.597756028029032 , \label{highroot} .
\end{eqnarray}
Setting $\gamma_{e,1}=0$ we have $\theta=0$ and 
\be
\omega_r^S = 0.999483839356918, \ \ \omega_h^S = 56.597756029072784 . \label{Sellroots} .
\ee
For the Lorentz model $\Gamma_V=0$ and so we need only evolve $\tilde{D}_W$ satisfying
\be
\f {\pa \tilde{D}_W}{\pa t} = -\gamma_{e,1} \left( \ba{c} K_x \\ K_y \ea \right) . \label{ddwdt}
\ee
The implicit Nordsieck method we use to evolve (\ref{ddwdt}) employs polynomials of degree $6$ which limits
the formal method order to $6$ in time. Nonetheless we will use values of $m$ ranging from $2$ to $6$ to
understand how the discretization of the dissipative term affects accuracy. 

Obviously, the high-frequency solutions, $H_z$, of the Sellmeier model and the Lorentz model will be approximately equal
up to $T=100$; their maximum relative difference is approximately $5 \times 10^{-4}$.
However, we will see that our method performs differently in these two cases. 

\subsection{Resonant Case}

As in the dielectric case we test for convergence with $m$ varying from $3$ to $6$. We fixed ${\rm CFL}=0.9$, which led to stable
results for the meshes tested. The meshes were chosen to have errors around $1\%$ for the coarsest mesh
and the mesh sizes changed by a convenient factor. Precisely for $m=3$ we use $N_G=100:25:250$, for $m=4$ $N_G=50:25:200$, for $m=5$
$N_G=50:25:200$, and for $m=6$ $N_G=30:20:150$. 

The results, shown in Figure \ref{convfig-resonant}, are again choppy from mesh to mesh. We
display a linear least squares fit to the log of the error as a function of the log of ${\rm DOF}/\lambda$. This produces
convergence rate estimates shown in Table \ref{convtab-resonant}. Perhaps surprisingly, in almost all cases these meet or excede the
theoretical rate of $2m$ for a dielectric medium or Sellmeier model, exceding the theoretical rate of $6$ of the approximation
to the dissipative term in the Lorentz model. We suspect this is due the contrast between the spatial and temporal frequencies; even
with ${\rm CFL}=0.9$ we are somewhat overrresolved in time. 

\begin{figure}[htb]
\begin{center}
\includegraphics[width=0.99\textwidth,trim={0.65cm 0.0cm 1.0cm 0.2cm},clip]{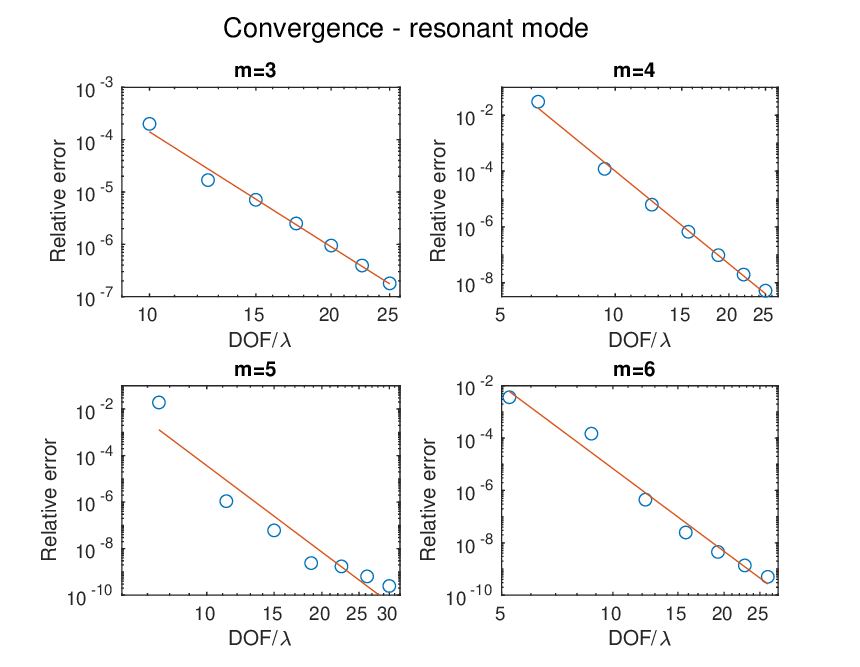}
\caption{Convergence for $m=3-6$, $k=40$, $T=100$ for the near-resonant
mode in a Lorentz medium. Errors
are computed by comparing the Hermite interpolant of the
numerical solution to the exact solution for each time step.\label{convfig-resonant}}
\end{center}
\end{figure} 

\begin{table}[htb]
\adjustbox{}{
\begin{tabular}{|c|c|r|} \hline
$m$ & DOF/$\lambda$ & Fit Rate \\ \hline \hline
$3$ & $10-25$ & $7.3$ \\ \hline
$4$ & $6.25-25$ & $11.0$  \\ \hline
$5$ & $7.5-30$ & $12.4$ \\ \hline 
$6$ & $5.25-26.25$ & $10.5$ \\ \hline \hline
\end{tabular}
}
\caption{Observed convergence for $m=3-6$ with $k=40$ and $T=100$ in the Lorentz model for the near-resonant mode. Here DOF/$\lambda$ denotes
the number of degrees-of-freedom per wavelength in each coordinate direction, 
$N_G(m+1)/k$ where the mesh is $N_G \times N_G$. That is $\Delta_x=\Delta_y = 2 \pi/N_G$. The error is
computed by (\ref{Edef}).\label{convtab-resonant}}
\end{table} 

The convergence arising from the various choices for $m$ is directly compared in Figure \ref{compfig-resonant}.
Again the results show that for any particular error level the larger values of $m$
are generally more efficient in terms of degrees-of-freedom required.

\begin{figure}[htb]
\begin{center}
\includegraphics[width=0.99\textwidth,trim={0.4cm 0.0cm 1.0cm 0.2cm},clip]{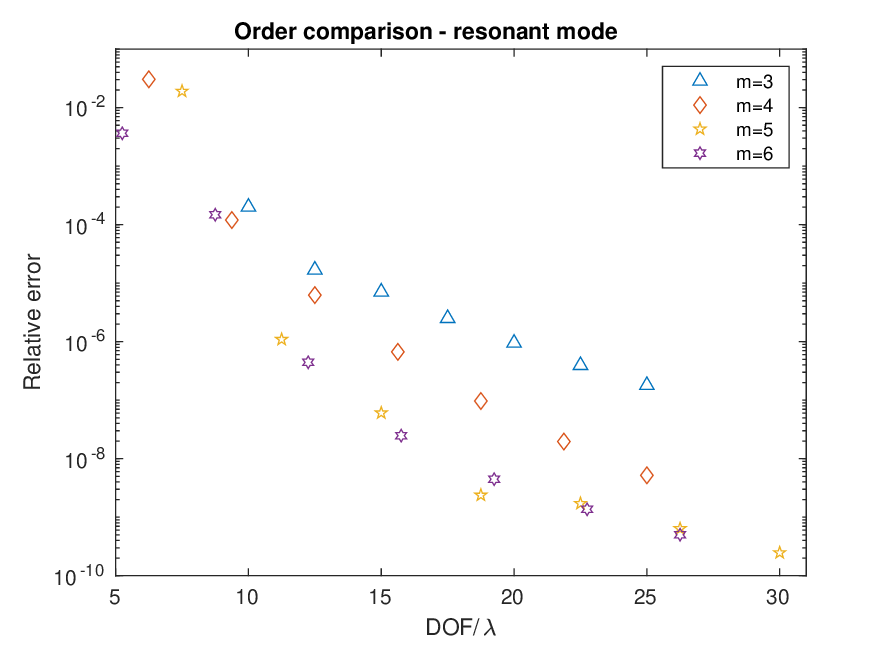}
\caption{Comparison of accuracy for $m=3-6$, $k=40$, $T=100$ in the Lorentz medium for
the near-resonant mode. Errors
are computed by comparing the Hermite interpolant of the
numerical solution to the exact solution for each time step.\label{compfig-resonant}}
\end{center}
\end{figure} 

\subsection{High-Frequency Case}

We now consider the high-frequncy solution of the Lorentz model with $k=40$. Recall that the dissipation is nearly
negligible in this case; the solution of the Sellmeier model obtained by setting $\gamma_{e,1}=0$ agrees with the solution
of the Lorentz model to more than three digits of accuracy, and so adding the dissipative term only improves accuracy if we
solve with a tolerance below $10^{-4}$. Nonetheless we find that the proposed method is limited by the sixth order treatment of
the dissipative terms, and that nothing is gained by increasing $m$ beyond $3$. In addition, comparing the results with $m=3$
for the Lorentz model to those shown below for the Sellmeier model we see that the mesh must be refined by more than $50\%$ to
achieve comparable accuracies.

In Figure \ref{convfig-highf} we examine convergence for $m=2$ and $m=3$. For $m=2$ we vary $N_G$ from $200$ to $800$ and
clearly observe sixth order convergence; the least-squares fit produces an estimated rate of $6.2$. For $m=3$ we vary $N_G$ from
$150$ to $450$ but only observe convergence for $N_G \geq 300$. The least squares fit to the last three data points yields
an estimated convergence rate of $7.5$. 

\begin{figure}[htb]
\begin{center}
\includegraphics[width=0.99\textwidth,trim={0.65cm 0.0cm 1.0cm 0.2cm},clip]{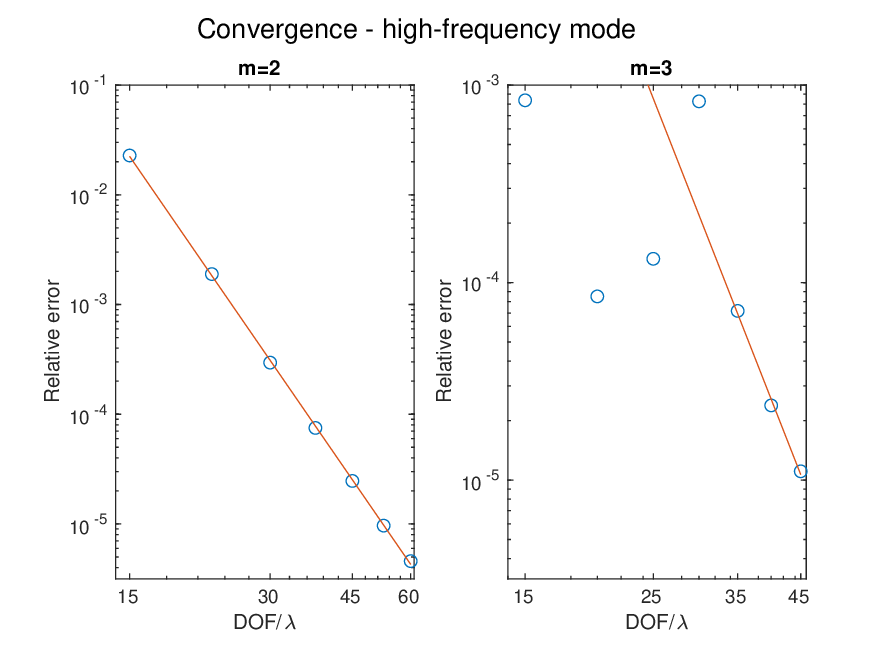}
\caption{Convergence for $m=2-3$, $k=40$, $T=100$ for the high-frequency
mode in a Lorentz medium. Errors
are computed by comparing the Hermite interpolant of the
numerical solution to the exact solution for each time step.\label{convfig-highf}}
\end{center}
\end{figure} 

The convergence arising for $m=2-4$ is directly compared in Figure \ref{compfig-highf}.
We observe that of the three choices $m=4$ is the least efficient in terms of
degrees-of-freedom required for a given tolerance, though there is some advantage to choosing $m=3$. 

\begin{figure}[htb]
\begin{center}
\includegraphics[width=0.99\textwidth,trim={0.5cm 0.0cm 1.0cm 0.2cm},clip]{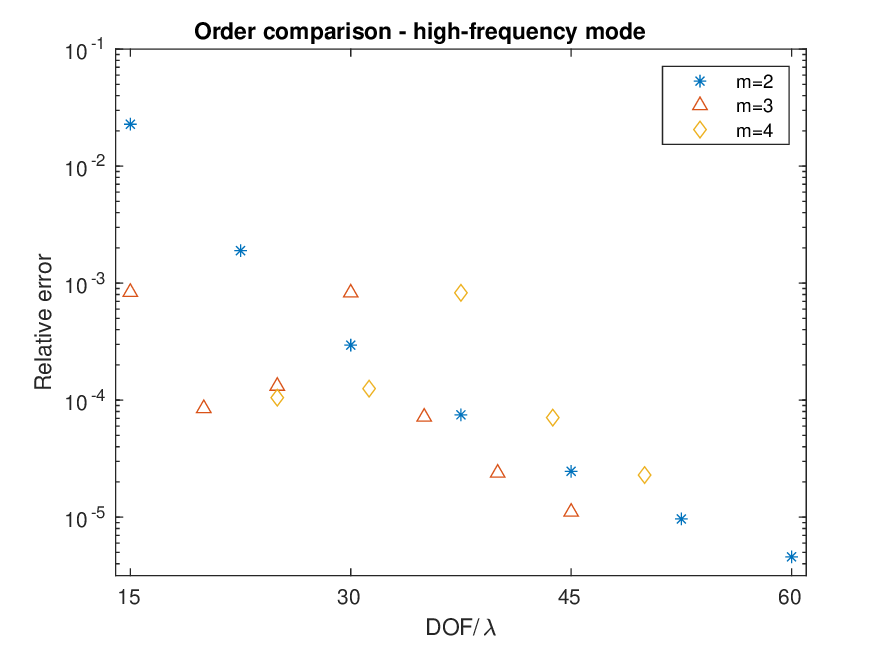}
\caption{Comparison of accuracy for $m=2-4$, $k=40$, $T=100$ in the Lorentz medium for
the high-frequency mode. Errors
are computed by comparing the Hermite interpolant of the
numerical solution to the exact solution for each time step.\label{compfig-highf}}
\end{center}
\end{figure}

\subsection{Sellmeier Model} 

Lastly we solve the Sellmeier model for the high-frequency mode. Here again we compare results for $m=3-6$. The mesh
sequences tested were $N_G=150:25:300$ for $m=3$, $100:25:250$ for $m=4$, $50:25:200$ for $m=5$, and
$40:20:160$ for $m=6$. 

The results, shown in Figure \ref{convfig-Sellmeier}, are very similar to the dielectric and resonant cases.
The least squares fit convergence rates produces
convergence rate estimates shown in Table \ref{convtab-Sellmeier}. In all cases these are at least $2m$.

\begin{figure}[htb]
\begin{center}
\includegraphics[width=0.99\textwidth,trim={0.6cm 0.0cm 1.0cm 0.2cm},clip]{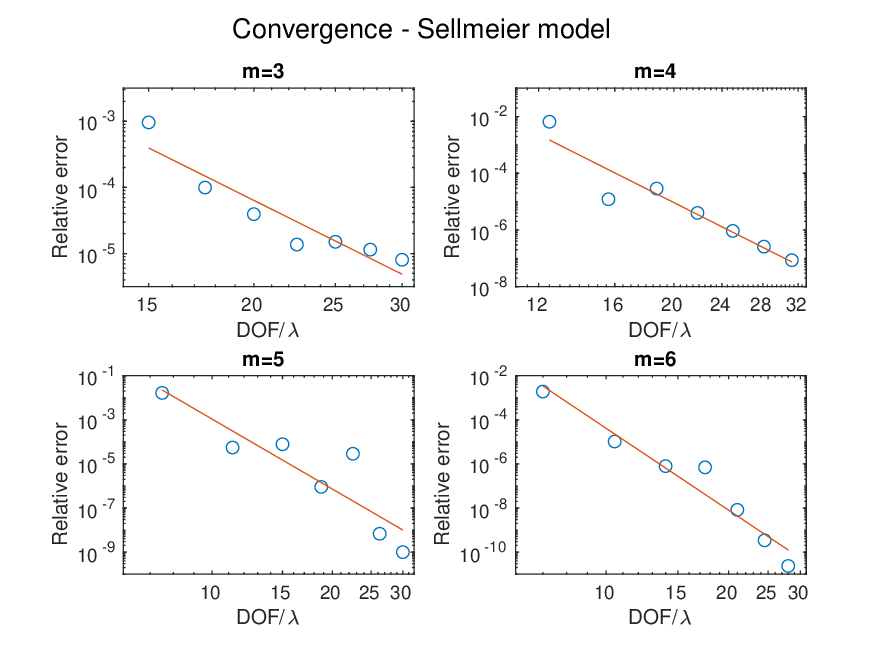}
\caption{Convergence for $m=3-6$, $k=40$, $T=100$ for the Sellmeier model.
Errors are computed by comparing the Hermite interpolant of the
numerical solution to the exact solution for each time step.\label{convfig-Sellmeier}}
\end{center}
\end{figure} 

\begin{table}[htb]
\adjustbox{}{
\begin{tabular}{|c|c|r|} \hline
$m$ & DOF/$\lambda$ & Fit Rate \\ \hline \hline
$3$ & $15-30$ & $6.3$ \\ \hline
$4$ & $12.5-31.25$ & $10.8$  \\ \hline
$5$ & $7.5-30$ & $10.6$ \\ \hline 
$6$ & $7-28$ & $12.4$ \\ \hline \hline
\end{tabular}
}
\caption{Observed convergence for $m=3-6$ with $k=40$ and $T=100$ for the Sellmeier model. Here DOF/$\lambda$ denotes
the number of degrees-of-freedom per wavelength in each coordinate direction, 
$N_G(m+1)/k$ where the mesh is $N_G \times N_G$. That is $\Delta_x=\Delta_y = 2 \pi/N_G$. The error is
computed by (\ref{Edef}).\label{convtab-Sellmeier}}
\end{table} 

The convergence arising from the various choices for $m$ is directly compared in Figure \ref{compfig-Sellmeier}.
Again the results show that for any particular error level the larger values of $m$
are generally more efficient in terms of degrees-of-freedom required.

\begin{figure}[htb]
\begin{center}
\includegraphics[width=0.99\textwidth,trim={0.4cm 0.0cm 1.0cm 0.2cm},clip]{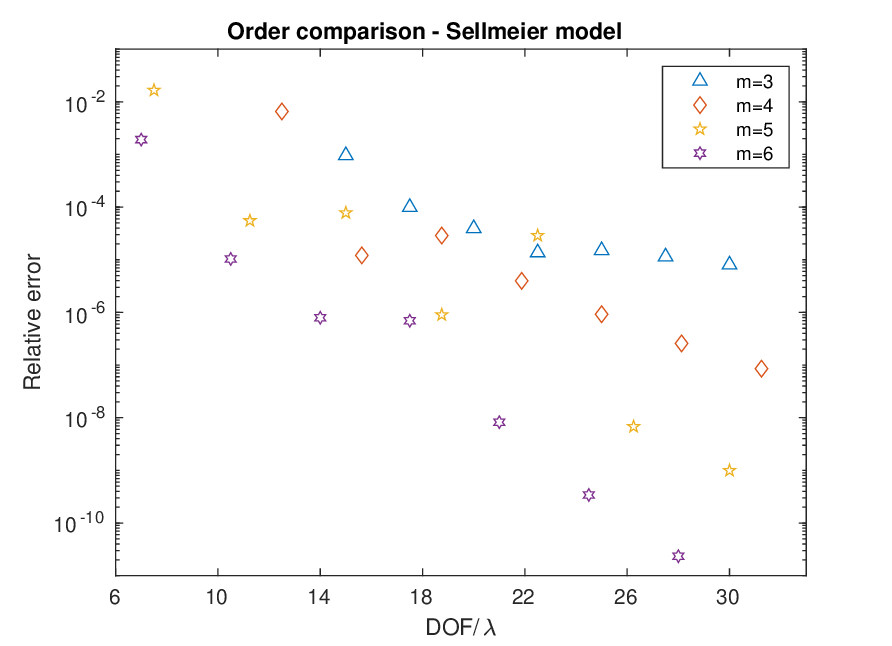}
\caption{Comparison of accuracy for $m=3-6$, $k=40$, $T=100$ 
for the Sellmeier model. Errors
are computed by comparing the Hermite interpolant of the
numerical solution to the exact solution for each time step.\label{compfig-Sellmeier}}
\end{center}
\end{figure}

\section{Conclusions and Open Issues}\label{sec:conclude} 

In conclusion we have proposed arbitrary-order energy-conserving Hermite discretizations of Maxwell's equations for both
dielectric and dissipation-free dispersive media. For these cases and with time-stepping of sufficiently high order we
prove stability for $\f {\Delta t}{\arrowvert \Delta_x \arrowvert} < 1$ independent of order. Numerical experiments show
that the high-order schemes are capable of accurately propagating waves over thousands of wavelengths with $9$ or fewer
degrees-of-freedom per wavelength. We also show how to include dissipation in the dispersive models, though this limits the
formal order of accuracy and, in some cases, signficantly degrades efficiency.

From a practical perspective, future work will focus on implementations in more complex geometry incorporating boundary and
interface conditions and on exploiting the locality of the evolution formulas for efficient implementation on current
computer archiectures. We will also consider if the possibility of using high-order dissipative Hermite methods \ci{GoodHer}
is worthwhile for Lorentz models.

In terms of theory, the fundamental open issues are a complete analysis of convergence in $L^2$ and of the stability of
the boundary and interface approximations.

\backmatter

\bmhead{Acknowledgments}

This work was funded in part by National Science Foundation Grants DMS-2012296, DMS-2309687 and DMS-2210286.
Any opinions, findings, and conclusions or recommendations expressed in this material are those of the authors and do not necessarily
reflect the views of the NSF.

\section*{Statements and Declarations}

\begin{description}
\item[Competing Interests]: The authors declare there are no competing interests.
\item[Code Availability]: The source codes used to produce the results in this paper are available by request to the corresponding author. 
\end{description}

\bibliography{hag}

\end{document}